\documentclass[11pt]{amsart}

\usepackage[a4paper,margin=1in]{geometry}
\usepackage{amsmath,amssymb,amsthm}
\usepackage{enumitem}
\usepackage[colorlinks=true,linkcolor=blue,citecolor=blue,urlcolor=blue]{hyperref}
\usepackage{orcidlink}
\usepackage{xcolor}
\usepackage{tikz}
\usetikzlibrary{decorations.pathreplacing}
\newtheorem{theorem}{Theorem}[section]
\newtheorem{proposition}[theorem]{Proposition}
\newtheorem{lemma}[theorem]{Lemma}
\newtheorem{corollary}[theorem]{Corollary}

\newtheorem{conjecture}{Conjecture}
\AtBeginDocument{%
  \setlength{\abovedisplayskip}{4pt plus 2pt minus 1pt}%
  \setlength{\belowdisplayskip}{4pt plus 2pt minus 1pt}%
  \setlength{\abovedisplayshortskip}{2pt plus 1pt minus 1pt}%
  \setlength{\belowdisplayshortskip}{3pt plus 1pt minus 1pt}%
}

\theoremstyle{definition}
\newtheorem{definition}[theorem]{Definition}
\theoremstyle{remark}
\newtheorem{remark}[theorem]{Remark}

\newcommand{\N}{\mathbb N}

\newcommand{\Z}{\mathbb Z}

\newcommand{\cB}{\mathcal B}
\newcommand{\cK}{\mathcal K}
\newcommand{\Fcal}{\mathcal F}
\newcommand{\M}{\mathcal M}
\newcommand{\mc}{\mathcal}
\newcommand{\mb}{\mathbb}
\newcommand{\mr}{\mathrm}
\newcommand{\Prob}{\operatorname{Prob}}
\newcommand{\supp}{\operatorname{supp}}
\newcommand{\esssup}{\operatorname*{ess\,sup}}
\newcommand{\tdeg}{\operatorname{tdeg}}

\newcommand{\diam}{\operatorname{diam}}
\newcommand{\id}{\operatorname{id}}
\newcommand{\Isom}{\operatorname{Isom}}

\newcommand{\Prof}{\operatorname{Prof}}
\newcommand{\IT}{\operatorname{IT}}
\newcommand{\IN}{\operatorname{IN}}
\newcommand{\Xeq}{X_{\mr{eq}}}
\newcommand{\Torus}{\mathbb T}
\newcommand{\multiset}[1]{\left\{\!\left\{#1\right\}\!\right\}}
\def\dist{\mr{dist}}
\title[Conditional Topomorphic Degree and Multivariate Mean Equicontinuity]{Conditional Topomorphic Degree and Multivariate Mean Equicontinuity for Minimal Amenable Group Actions}

\author[C. Liu]{Chunlin Liu\orcidlink{0000-0001-6277-013X}}
\address[C. Liu]{School of Mathematical Sciences, Dalian University of
Technology, Dalian 116024, P.R. China; and Institute of Mathematics,
Polish Academy of Sciences, ul. Śniadeckich 8, 00-656 Warszawa, Poland}
\email{chunlinliu@mail.ustc.edu.cn}

\thanks{This article  was supported by the
Postdoctoral Fellowship Program and China Postdoctoral Science Foundation under Grant Number BX20250067, and the China Postdoctoral Science Foundation under Grant Number 2025M773074.
}
\date{}

\begin{document}
\begin{abstract}
Let $G$ be a countably infinite discrete amenable group acting minimally
on a compact metric space $X$, and let $\pi:X\to X_{\mathrm{eq}}$ be the
maximal equicontinuous factor map.  We introduce the
\emph{conditional topomorphic degree}
$d:=\tdeg_G(X)\in\mathbb N\cup\{\infty\}$, which, when finite, is the
least integer such that $\pi$ is an at most $d$-to-one topomorphic
extension.  We prove that, for every $r\ge2$, Weyl mean
$r$-equicontinuity, mean $r$-equicontinuity along some F{\o}lner
sequence, and $d\le r-1$ are equivalent.  For minimal
$\mathbb Z$-systems, this settles a conjecture of Breitenb\"ucher,
Haupt, and J\"ager.

We further establish the decomposition formula $d=\sum_{\mu\in\mathcal M_G^e(X)}\iota_\mu\exp\bigl(h_\mu^*(G)\bigr),$
where $\iota_\mu$ is the degree of the factor map from the
measure-theoretic maximal compact factor associated with $\mu$ onto
$X_{\mathrm{eq}}$, and $h_\mu^*(G)$ is the maximal measure sequence
entropy of $\mu$. As further consequences of the decomposition formula, we show that for every finite $N$ with
$2\le N\le d$, the system admits an essential IT $N$-tuple.
Consequently, $h_{\mathrm{top}}^*(X,G)\ge \log d.$
This strengthens a lower bound of Liu, Wang, and Xu by also detecting
compact multiplicities.

As an application, we answer a question of G\'omez, Le\'on-Torres,
and Mu\~noz-L\'opez.  If $G$ contains a finite-index normal subgroup
isomorphic to $\mathbb Z^r$, then, for every $m\ge2$, there exists a
free minimal uniquely ergodic zero-entropy finite-alphabet $G$-subshift
with maximal topological sequence entropy $\log m$, an essential IT
$m$-tuple, and no essential IN $(m+1)$-tuple.  For $G=\mathbb Z$, the
alphabet may be chosen to have exactly $m$ symbols.

Finally, we realize every finite multiplicity profile by a zero-entropy
minimal almost one-to-one extension of an irrational circle rotation.
\end{abstract}

\maketitle

\section{Introduction}
\label{sec:introduction}

Classical equicontinuity requires nearby points to remain uniformly close
at every group time.  Mean equicontinuity weakens this demand by averaging
the orbit distance, thereby allowing a set of exceptional times of small
density.  This idea goes back to mean-$L$-stability
\cite{Auslander1959} and was developed by Li,
Tu and Ye \cite{LiTuYe2015}.  Their work established, among other things,
the mean-equicontinuity and mean-sensitivity dichotomy in the minimal and
transitive settings and the discreteness of the spectrum of every ergodic
measure carried by a mean equicontinuous system.  Subsequent work connected
mean equicontinuity with bounded measure complexity, almost automorphy,
sequence entropy and continuous spectral data; see
\cite{GarciaRamos2017,GarciaRamosJaegerYe2021,
GarciaRamosMarcus2019,HuangEtAl2021,li2021meanEquicontinuity}.

For minimal systems this averaged orbit condition has a particularly rigid
extension-theoretic interpretation.  Downarowicz and Glasner proved for
$\mathbb Z$-actions that mean equicontinuity is equivalent to the factor
map onto the maximal equicontinuous factor being a measure-theoretic
isomorphism \cite{DownarowiczGlasner2016}.  Fuhrmann, Gr\"oger and Lenz
extended this picture to locally compact, $\sigma$-compact amenable group
actions: in the minimal setting, mean equicontinuity is equivalent to
topo-isomorphy over the maximal equicontinuous factor and also to discrete
spectrum with continuous eigenfunctions
\cite{FuhrmannGroegerLenz2022}.  Thus pairwise mean equicontinuity is not
merely a regularity property of orbit averages; it says that the maximal
equicontinuous factor loses no information from the viewpoint of invariant
measures.

The pairwise theory distinguishes the case in which almost every
conditional fiber consists of one effective point, but it does not quantify
higher finite multiplicities.  Finite-multiplicity extensions of
equicontinuous systems arise naturally in substitution systems, Toeplitz
flows, and model-set dynamics; see
\cite{BaakeGrimm2013,Downarowicz2005,
IwanikLacroix1994,FuhrmannGlasnerJaegerOertel2021}.
The multivariate theory is designed precisely to capture such phenomena.
Mean-sensitive tuples and related multivariate sensitivity notions were
studied by Li, Ye and Yu \cite{LiYeYu2022,LiYu2021}.  Their connections with sequence
entropy tuples were established in
\cite{GarciaRamosMunozLopez2024,LiLiuTuYu2024,LiuXuZhang2025Measure}. For amenable group actions,
related mean-sensitive and regional mean-sensitive notions were developed
in \cite{LiuYin2023,HauserLiu2026Regional}.
Other multivariate equicontinuity notions were used by Garc\'ia-Ramos and
Le\'on-Torres to characterize the maximal rank and coincidence rank of the
maximal equicontinuous factor
\cite{GarciaRamosLeonTorres2025}.
Related notions of multivariate diam-mean equicontinuity and frequent
stability, together with their structural and local characterizations,
were investigated in
\cite{Haupt2025,HauptJaegerLiu2025}.  Breitenb\"ucher, Haupt and J\"ager
introduced mean $r$-equicontinuity together with finite-to-one
topomorphic extensions \cite{BreitenbuecherHauptJaeger2026}.  They proved
that an at most $m$-to-one topomorphic extension of the maximal
equicontinuous factor is mean $(m+1)$-equicontinuous and formulated the
converse as follows.

\begin{conjecture}\label{conj1}
For $m\in\N$ with $m\ge 2$, a minimal $\mathbb Z$-system is mean
$(m+1)$-equicontinuous but not mean $m$-equicontinuous if and only
if it is an $m$-to-one topomorphic extension of its maximal
equicontinuous factor.
\end{conjecture}

Our main result shows this conjecture holds for all countably infinite discrete amenable groups. More precisely,
let $G$ be a countably
infinite discrete amenable group, let $(X,G)$ be a minimal action on a compact
metric space $(X,d)$, and let
$
\pi:(X,G)\to(\Xeq,G)
$
be its maximal equicontinuous factor.  The system $(\Xeq,G)$ is
uniquely ergodic; write $\nu$ for its invariant measure.  If
$\mu\in\M_G(X)$, disintegrate
$
\mu=\int_{\Xeq}\mu_y\,d\nu(y)
$
over $\pi$.  We define the \emph{conditional topomorphic degree} by
\begin{equation}\label{eq:intro-topomorphic-degree}
\tdeg_G(X)
:=
\sup_{\mu\in\M_G(X)}
\esssup_{y\in\Xeq}|\supp(\mu_y)|
\in\N\cup\{\infty\},
\end{equation}
where every infinite support is assigned the value $\infty$.
 For $\mathbb Z$-actions, the inequality
$\tdeg_{\mathbb Z}(X)\le m$ is precisely the conditional-measure
characterization of an at most $m$-to-one topomorphic extension
\cite[Proposition~3.1]{BreitenbuecherHauptJaeger2026}.
In fact, the same characterization remains valid for actions of
arbitrary countable discrete groups;
see Appendix~\ref{app:finite-topomorphic-characterization}.

\begin{theorem}
\label{thm:main}
Let $r\ge2$.  The following are equivalent.
\begin{enumerate}[label=\textup{(\roman*)},leftmargin=2.4em]
\item $(X,G)$ is Weyl mean $r$-equicontinuous.
\item $(X,G)$ is $\Fcal$-mean $r$-equicontinuous for some 
F{\o}lner sequence $\Fcal$.
\item $\tdeg_G(X)\le r-1$.
\end{enumerate}
\end{theorem}

The implication \textup{(ii)}$\Rightarrow$\textup{(iii)} is the main
difficulty.  The precise obstruction was already clearly identified by
Breitenb\"ucher, Haupt and J\"ager
\cite[Remark~5.2]{BreitenbuecherHauptJaeger2026}.  If a conditional measure
over the maximal equicontinuous factor has $r$ separated points in its
support, then these points lie in a single fiber.  Mean
$r$-equicontinuity, however, is tested on $r$ points chosen in an
arbitrarily small common neighborhood.  Thus the measure-theoretic
configuration supplied by the conditional measure does not directly
produce the local topological configuration needed to contradict mean
$r$-equicontinuity.  In particular, there is in general no group element
that moves the original fiber tuple into an arbitrarily small common
neighborhood.

Our argument overcomes this obstruction by transporting the compact
phase data rather than the original fiber tuple.  The separated points
in the conditional fiber are used only to choose pairwise separated
target neighborhoods
$
V_1,\ldots,V_r
$
and the corresponding compatible phases in the measure-theoretic
maximal compact factors.  Fix a neighborhood $U$ of the point at
which mean $r$-equicontinuity is tested.  Positivity of the target
rectangle
$
V_1\times\cdots\times V_r
$
persists under small perturbations of the phase parameter.  On the
other hand, compatibility of the prescribed phases with the source
point, together with Haar averaging over the joint phase group, shows
that a positive-measure set of such phase parameters gives positive
mass to $U^r$.  Since almost every lifted phase joining is ergodic, we
can therefore choose a phase parameter $t$ such that
\[
\Lambda_t(U^r)>0
\qquad\text{and}\qquad
\Lambda_t(V_1\times\cdots\times V_r)>0.
\]
A generic point
$
(y_1,\ldots,y_r)\in U^r
$
for $\Lambda_t$ then enters
$V_1\times\cdots\times V_r$ with positive upper density along the
prescribed F{\o}lner sequence.  Since the target neighborhoods are
uniformly separated, this tuple has positive mean $r$-distance.
Thus no group element is required to move the original fiber tuple
into $U$; the mixed compact-phase joining recreates the same separated
orbit pattern inside $U$; see Theorem \ref{thm:mixed-phase-local-realization}.

The compact-factor analysis also yields an exact decomposition of the
conditional topomorphic degree.  For
$\mu\in\M_G^e(X)$, let
$
\kappa_\mu:(X,\mu,G)\to(Z_\mu,m_\mu,G)
$
be the maximal compact factor of $(X,\mu,G)$.  The maximal
equicontinuous factor map $\pi$ factors through $\kappa_\mu$
measure-theoretically: there exists a continuous factor map
$
p_\mu:Z_\mu\to\Xeq
$
such that
\[
\pi=p_\mu\circ\kappa_\mu
\qquad \mu\text{-almost everywhere (a.e.)}.
\]
Let $\iota_\mu$ denote the common cardinality of the fibers of
$p_\mu$, and let $b_\mu$ denote the almost everywhere constant
cardinality of the support of the disintegration of $\mu$ over
$Z_\mu$.  Thus $\iota_\mu$ counts the compact phases that are
identified by the topological maximal equicontinuous factor, while
$b_\mu$ measures the residual multiplicity after the compact phase
has been fixed. 
The two multiplicities combine to give the following exact formula.

\begin{theorem}\label{thm:intro-degree}
For every $\mu\in\M_G^e(X)$,
\[
 |\supp(\mu_y)|
 =\iota_\mu b_\mu
 =\iota_\mu\exp\bigl(h_\mu^*(G)\bigr)
 \quad\text{for $\nu$-a.e. }y\in\Xeq.
\]
Moreover,
\[
 \tdeg_G(X)
 =\sum_{\mu\in\M_G^e(X)}\iota_\mu b_\mu
 =\sum_{\mu\in\M_G^e(X)}
 \iota_\mu\exp\bigl(h_\mu^*(G)\bigr),
\]
where the sums are understood in the extended sense.
\end{theorem}
In the finite case, the multiplicative contribution of a fixed
ergodic measure $\mu$ can be read fiberwise as follows:
\begin{center}
\begin{tikzpicture}[
  >=stealth,
  every node/.style={font=\small},
  measurable/.style={->,dashed,thin},
  factor/.style={->,thin}
]

\node (A1) at (-2.5,2.2)
  {$\underbrace{\bullet\;\cdots\;\bullet}_{b_\mu}$};
\node at (0,2.2) {$\cdots$};
\node (Ai) at (2.5,2.2)
  {$\underbrace{\bullet\;\cdots\;\bullet}_{b_\mu}$};

\node (z1) at (-2.5,1.0) {$z_{\mu,1}$};
\node at (0,1.0) {$\cdots$};
\node (zi) at (2.5,1.0) {$z_{\mu,\iota_\mu}$};

\node (y) at (0,-0.2) {$y\in\Xeq$};

\draw[measurable]
  (A1) -- node[left,font=\scriptsize] {$\kappa_\mu$} (z1);
\draw[measurable] (Ai) -- (zi);

\draw[factor]
  (z1) -- node[below left,font=\scriptsize] {$p_\mu$} (y);
\draw[factor] (zi) -- (y);

\draw[
  decorate,
  decoration={brace,amplitude=4pt}
]
  (-3.1,2.8) -- (3.1,2.8)
  node[midway,above=4pt]
  {$\iota_\mu$ compact phases};

\node at (0,-1.05)
  {$\displaystyle
    |\supp(\mu_y)|
    =
    \underbrace{b_\mu+\cdots+b_\mu}_{\iota_\mu\ \mathrm{terms}}
    =
    \iota_\mu b_\mu
    =
    \iota_\mu\exp\!\bigl(h_\mu^*(G)\bigr)$};

\end{tikzpicture}
\end{center}

\begin{remark}
\label{rem:unique-arithmetic-rigidity}
By \cite[Theorem~2.9]{LiuWangXu2025}, for every
$\mu\in\mc M_G^e(X)$ there exists
$b_\mu\in\N\cup\{\infty\}$ such that
\[
h_\mu^*(G)=\log b_\mu.
\]

Now suppose that $(X,G)$ is uniquely ergodic, with unique invariant
measure $\mu$, and that
$
d:=\tdeg_G(X)<\infty.
$
Then Theorem~\ref{thm:intro-degree} implies
\[
h_\mu^*(G)\in\{\log b:b\mid d\}.
\]
\end{remark}

The notion of an IT-tuple originates in the independence theory of
Kerr and Li \cite{KerrLi2007}; see also \cite{Huang2006}.
Maximal pattern entropy and its tuple-theoretic description were
developed in \cite{Goodman1974,HuangYe2009}.
The case $\tdeg_G(X)=1$ is closely connected with earlier
fiber-rigidity results formulated in terms of sequence entropy and
independence.  For $\mathbb Z$-actions, a minimal system with zero
maximal pattern entropy is almost automorphic
\cite{HuangLiShaoYe2003}, and its maximal equicontinuous factor map is
regular one-to-one
\cite{GarciaRamos2017,GarciaRamosJaegerYe2021}.
For minimal group actions without essential IT-pairs,
analogous almost one-to-one and regularity results were obtained in
\cite{Glasner2018Tame,FuhrmannGlasnerJaegerOertel2021}.
Higher finite-multiplicity fiber-rigidity results were established in
\cite{HuangLianShaoYe2021,LiuXuZhang2025Minimal}.  The degree decomposition above yields
the following conclusion.
\begin{theorem}
\label{thm:intro-it-tuples}
For every finite integer $N$ with
$
2\le N\le d,
$
the system $(X,G)$ admits an essential IT $N$-tuple.  Moreover,
\[
h_{\mathrm{top}}^*(X,G)\ge\log \tdeg_G(X),
\]
where $\log\infty:=\infty$.  
\end{theorem}
\begin{remark}
This improves the variational lower bound established in
\cite{LiuWangXu2025},
\[
h_{\mathrm{top}}^*(X,G)
\ge
\log\sum_{\mu\in\M_G^e(X)}
\exp\bigl(h_\mu^*(G)\bigr)
=
\log\sum_{\mu\in\M_G^e(X)}b_\mu.
\]
Indeed, that estimate detects only the residual multiplicities
$b_\mu$ remaining after the maximal compact phase has been fixed.  The
present argument also captures the compact multiplicities
$\iota_\mu$.  
\end{remark}

Theorem~\ref{thm:intro-it-tuples} also yields a new application to the realization
problem for finite sequence entropy. G\'omez, Le\'on-Torres and
Mu\~noz-L\'opez constructed free minimal zero-entropy subshifts with
finite positive maximal topological sequence entropy for groups
virtually isomorphic to $\mathbb Z^p$.  In the one-dimensional case,
they obtained the stronger exact realization $h_{\mathrm{top}}^*=\log m$
on an $m$-symbol subshift.  They asked whether both results admit
uniquely ergodic versions
\cite[Question~5.1]{GomezLeonTorresMunozLopez2026}.  We answer both
parts affirmatively and, in the virtually $\mathbb Z^p$ case, obtain
the stronger exact value $\log m$.

The two constructions exhibit distinct mechanisms for generating
independence.  In the Toeplitz examples of G\'omez, Le\'on-Torres and
Mu\~noz-L\'opez, the relevant multiplicity is manifested through the
presence of several ergodic invariant measures; in particular, for
$\mathbb Z$-actions, $m$ ergodic measures yield an essential IT
$m$-tuple.  The degree decomposition identifies an additional source
of multiplicity, namely the conditional topomorphic multiplicity
carried by a single ergodic invariant measure.  By
Theorem~\ref{thm:intro-it-tuples}, this internal multiplicity already
forces IT-tuples, and hence remains effective in the uniquely ergodic
setting.  This leads to uniquely ergodic realizations with prescribed
maximal topological sequence entropy $\log m$ for virtually
$\mathbb Z^p$ groups.

\begin{theorem}
\label{thm:intro-unique-virtual-abelian}
Let $G$ be a countable group containing a finite-index normal subgroup
$H\cong\mathbb Z^p$ for some $p\in\N$.  For every $m\ge2$, there are a
finite alphabet $\mathcal A$ and a free minimal uniquely ergodic
subshift
$
 X_{m,G}\subseteq\mathcal A^G
$
such that $(X_{m,G},G)$ has zero topological entropy and
\[
h_{\mathrm{top}}^*(X_{m,G},G)=\log m.
\]
Moreover, $(X_{m,G},G)$ admits an essential IT $m$-tuple and has no
essential IN $(m+1)$-tuple. When $G=\mathbb Z$, one may take
$
\mathcal A=\{0,1,\ldots,m-1\}.
$
\end{theorem}

We next turn to the inverse realization problem suggested by the
degree decomposition.
Theorem \ref{thm:intro-degree} separates the compact and residual contributions to
the conditional topomorphic degree.  Our final main result shows that, in finite degree, every finite multiset of positive integer pairs can be realized,
even within the rigid class of zero-entropy minimal almost one-to-one
extensions of irrational circle rotations.

For a finite-degree minimal $\mathbb Z$-system $(X,T)$, define its
\emph{multiplicity profile} to be the indexed family
\[
\Prof(X,T)
:=
\bigl((\iota_\mu,b_\mu)\bigr)_{\mu\in\M_T^e(X)},
\]
viewed as a multiset of pairs of positive integers\footnote{This means that if two distinct ergodic measures
$\mu,\nu\in\M_T^e(X)$ satisfy
$(\iota_\mu,b_\mu)=(\iota_\nu,b_\nu)$, then this pair occurs twice in
$\Prof(X,T)$.}.

\begin{theorem}
\label{thm:finite-profile-realization}
For every finite multiset
$
\mathcal P
=
\multiset{
(\iota_1,b_1),\ldots,(\iota_\ell,b_\ell)
}
$
of pairs of positive integers, there exists a zero-entropy minimal
$\mathbb Z$-system $(X,T)$, which is an almost one-to-one extension
of an irrational circle rotation, such that
\[
\M_T^e(X)=\{\mu_1,\ldots,\mu_\ell\},
\qquad
\Prof(X,T)=\mathcal P.
\]
\end{theorem}
The construction combines the finite-multiplicity extensions of
irrational rotations developed in \cite{HauptJaeger2025} with the
relative Furstenberg--Weiss theorem of
\cite{DownarowiczLacroix1998}.

\subsection*{The paper is organized as follows.}
Section~\ref{sec:preliminaries} collects the necessary preliminaries.
Section~\ref{sec:compact-phase-decompositions} develops mixed
compact-phase joinings, proves the local realization theorem,
Theorem~\ref{thm:mixed-phase-local-realization}, and derives the
multiplicity obstructions needed for the converse implication of the
main theorem.
Section~\ref{sec:main-proof}  proves
Theorems~\ref{thm:main} and~\ref{thm:intro-degree}.
Section~\ref{sec:sequence-entropy}
proves Theorem~\ref{thm:intro-it-tuples} and derives the uniquely
ergodic realization in
Theorem~\ref{thm:intro-unique-virtual-abelian}.
Finally, Section~\ref{sec:realization} proves
Theorem~\ref{thm:finite-profile-realization} and derives a characterization in the uniquely ergodic case.

\section{Preliminaries}
\label{sec:preliminaries}

\subsection{Basic notation and conventions}
Throughout the paper,
$
\mb N:=\{1,2,\ldots\},
$
and, for $N\in\mb N$, we write
$
[N]:=\{1,\ldots,N\}.
$ We use the conventions
\[
n\cdot\infty=\infty\cdot n=\infty,\qquad
\exp(\infty)=\infty,\qquad\text{for every $n\in\mathbb N$.}
\]

We always assume that $G$ is a countably infinite discrete amenable group. A sequence
$
\Fcal=(F_n)_{n\ge1}
$
of nonempty finite subsets of $G$ is called a \emph{(left) F{\o}lner
sequence} if
\[
\lim_{n\to\infty}\frac{|hF_n\triangle F_n|}{|F_n|}
= 0
\qquad
\text{for every }h\in G.
\]
Since $G$ is amenable,  F{\o}lner sequences exist.

Let $X$ be a compact metric space and put
\[
\Delta^N_{\mr{fat}}(X)
:=
\left\{
(x_1,\ldots,x_N)\in X^N:
x_i=x_j\text{ for some }i<j
\right\}.
\]
A tuple in $X^N\setminus\Delta_{\mathrm{fat}}^N(X)$ is called \emph{essential}. 

For topological spaces $X$ and $Z$, we denote by
\[
\operatorname{pr}_X:X\times Z\to X,
\qquad
\operatorname{pr}_Z:X\times Z\to Z
\]
the canonical coordinate projections.

A \emph{$G$-system} is a compact metric space $X$
equipped with a continuous action.  
A $G$-system $(X,G)$ is called \emph{equicontinuous} if  for every $\varepsilon>0$, there exists $\delta>0$ such that \[ d_X(x,x')<\delta \quad\Longrightarrow\quad d_X(gx,gx')<\varepsilon \qquad\text{for all }g\in G. \] For a compact metrizable $G$-system, equicontinuity is equivalent to the existence of a compatible $G$-invariant metric. 

We write $\Prob(X)$ for the set of Borel probability measures on
$X$,
$\M_G(X)$ for the set of $G$-invariant Borel probability measures on
$X$, and $\M_G^e(X)$ for the subset of ergodic measures. Since $G$ is amenable, the set $\M_G^e(X)$ is nonempty.

A \emph{probability-measure-preserving $G$-system}, abbreviated as a
\emph{p.m.p.\ $G$-system}, is a standard probability space
$
(Z,\mc B_Z,m_Z)
$
equipped with a measurable action
such that
\[
g_*m_Z=m_Z
\qquad
\text{for every }g\in G.
\]
In particular, each $G$-system with an invariant measure is a p.m.p. $G$-system.
For a family $(a_\mu)_{\mu\in\M_G^e(X)}\subseteq\N\cup\{\infty\}$, we use
throughout the extended-sum convention
$$
 \sum_{\mu\in\M_G^e(X)}a_\mu
 :=
 \sup\left\{
 \sum_{\mu\in F}a_\mu:
 F\subseteq\M_G^e(X)\text{ finite}
 \right\}.
$$

We use both topological and measure-theoretic factor maps:
\begin{enumerate}[label=\textup{(\roman*)},leftmargin=2.2em]
\item
Let $(Z,G)$ and $(Y,G)$ be $G$-systems.
A continuous surjection
$
p:Z\to Y
$
is called a \emph{topological factor map} if
\[
p(gz)=gp(z)
\qquad
\text{for every }g\in G\text{ and }z\in Z.
\]

\item
Let $(Z,m_Z,G)$ and $(Y,m_Y,G)$ be
p.m.p. $G$-systems.
A measurable map
$
p^0:Z\to Y
$
is called a \emph{measure-theoretic factor map} if
\[
(p^0)_*m_Z=m_Y
\qquad\text{and}\qquad
p^0(gz)=gp^0(z)
\quad
\text{for }m_Z\text{-a.e. }z\in Z\text{ and for every $g\in G$}.
\]
\end{enumerate}

\subsection{Averaging along F{\o}lner sequences}

Let $(X,G)$ be a $G$-system with $\mu\in\mc M_G(X)$. For $f\in L^2(\mu)$ and a F\o lner sequence
$\Fcal=(F_n)_{n\ge1}$, write
\[
 A_n^{\Fcal}f(x):=\frac1{|F_n|}\sum_{g\in F_n}f(gx).
\]

We use the mean ergodic theorem for amenable group actions; see, for
example, \cite[Theorem~5.5]{BergelsonEtAl2006}.
\begin{lemma}
\label{lem:mean-ergodic} The average
$A_n^{\Fcal}f$ converges in $L^2(\mu)$ to
$\mathbb E_\mu(f\mid\mc I_\mu)$ for every $f\in L^2(\mu)$ and every 
F{\o}lner sequence $\Fcal$.  In particular, if $\mu$ is ergodic, then
\[
\lim_{n\to\infty} A_n^{\Fcal}f=\int f\,d\mu
 \qquad\text{in }L^2(\mu).
\]
Here $\mc I_\mu$ is the $\sigma$-algebra generated by all invariant measurable sets.
\end{lemma}
As a direct corollary, we have the following.
\begin{lemma}
\label{lem:uniform-lower-bound}
Let $f\in C(X)$ and suppose that
\[
 \int f\,d\mu\ge c
 \qquad
 \text{for every }\mu\in\M_G(X).
\]
Then, for every $x\in X$ and every F{\o}lner sequence $\Fcal$,
\[
 \liminf_{n\to\infty}A_n^{\Fcal}f(x)\ge c.
\]
\end{lemma}

\begin{proof}
If the conclusion failed, pass to a subsequence along which the averages
converge to a number smaller than $c$.  A weak* cluster point of
$
 \frac1{|F_n|}\sum_{g\in F_n}\delta_{gx}
$
is $G$-invariant, because
$|hF_n\triangle F_n|/|F_n|\to0$ for every $h\in G$.  Its integral of $f$
is smaller than $c$, a contradiction.
\end{proof}

\subsection{Hyperspaces and compact-valued maps}
\label{subsec:hyperspaces}

Let $(Z,d_Z)$ be a compact metric space.  We denote by
$\mc H(Z)$ the hyperspace of all nonempty closed subsets of $Z$.
Since $Z$ is compact, these are precisely its nonempty compact subsets.
For $C,D\in\mc H(Z)$, their Hausdorff distance is
\[
d_{\mathrm H}(C,D)
:=
\max\left\{
\sup_{z\in C}d_Z(z,D),
\sup_{w\in D}d_Z(w,C)
\right\}.
\]
The metric space $(\mc H(Z),d_{\mathrm H})$ is compact.

A $G$-system $(Z,G)$ induces an action on
$\mc H(Z)$ by
\[
gC:=\{gz:z\in C\}.
\]
If $d_Z$ is $G$-invariant, then the induced action on
$\mc H(Z)$ is also isometric:
\[
d_{\mathrm H}(gC,gD)=d_{\mathrm H}(C,D)
\qquad
(g\in G,\ C,D\in\mc H(Z)).
\]

A compact-valued map
$
A:X\to\mc H(Z)
$
is called equivariant if
\[
A(gx)=gA(x)
\qquad
(g\in G,\ x\in X).
\]
Equivalently, its graph
\[
\Gamma_A
:=
\{(x,z)\in X\times Z:z\in A(x)\}
\]
is $G$-invariant.  A closed graph guarantees upper semicontinuity of
$A$, but does not in general imply Hausdorff continuity.  The following
lemma shows that equivariance upgrades upper semicontinuity to Hausdorff
continuity when the source system is minimal and the action on the range is
isometric. 
\begin{lemma}
\label{lem:automatic-fiber-continuity}
Let $(X,G)$ be a minimal $G$-system, let $(Z,G)$ be a $G$-system with a $G$-invariant metric, and let
$
\Gamma\subseteq X\times Z
$
be a closed $G$-invariant set whose first projection is $X$. 
Then the map
\[
A:X\to\mc H(Z),\qquad x\mapsto A(x):=\{z\in Z:(x,z)\in\Gamma\}
\]
is continuous and equivariant.
\end{lemma}

\begin{proof}
Equivariance follows from the invariance of $\Gamma$.  Fix
$\varepsilon>0$, and let $\Omega_\varepsilon$ be the set of points
$x\in X$ for which some open neighborhood $U\ni x$ satisfies
\begin{equation}\label{eq:automatic-small-oscillation}
d_{\mathrm H}(A(y),A(y'))<\varepsilon
\qquad(y,y'\in U).
\end{equation}
The set $\Omega_\varepsilon$ is open.  We show that it is nonempty.

Choose compact sets $C_1,\ldots,C_m\subseteq Z$ such that
\[
Z=\bigcup_{j=1}^m C_j,
\qquad
\diam(C_j)<\varepsilon.
\]
For every $j\in[m]$, put
\[
H_j:=\{x\in X:A(x)\cap C_j\ne\emptyset\}.
\]
Each $H_j$ is closed as
$
H_j=\operatorname{pr}_X\bigl(\Gamma\cap(X\times C_j)\bigr).
$
Starting from any nonempty open set $U_0\subseteq X$, recursively choose
nonempty open sets
$
U_m\subseteq\cdots\subseteq U_1\subseteq U_0
$
so that, for each $j$, either $U_j\subseteq H_j$ or
$U_j\cap H_j=\emptyset$.  Put $V:=U_m$. 
For any
 $x,y\in V$ and
$z\in A(x)$, we may choose $j$ with $z\in C_j$, and so $x\in H_j$.  The
chosen alternative forces $V\subseteq H_j$, and hence some
$z'\in A(y)\cap C_j$, and $d_Z(z,z')<\varepsilon$.  Interchanging
$x$ and $y$ gives
$
d_{\mathrm H}(A(x),A(y))<\varepsilon.
$
Thus $\Omega_\varepsilon\ne\emptyset$.

As the action on $Z$  is isometric,
$\Omega_\varepsilon$ is $G$-invariant.  Since $\Omega_\varepsilon$ is nonempty, open, and $G$-invariant,
minimality gives $\Omega_\varepsilon=X$.
\end{proof}

\subsection{Compact homogeneous models}

We recall the compact homogeneous models needed in both the topological
and measure-theoretic categories.
\subsubsection{Topological compact models}
The compact homogeneous-space representation of minimal equicontinuous
systems is classical;  see,
for example, \cite[Chapter~2]{Auslander1988} and
\cite{Glasner2003}. We include the proofs in order to record the
precise description of the conditional measures on factor fibers and the
measurable-to-topological factor principle needed below.

\begin{lemma}
\label{lem:homogeneous-factor-fibers}
Let $(Z,G)$ be a minimal equicontinuous system.
\begin{enumerate}[label=\textup{(\alph*)},leftmargin=2.2em]
\item There are a compact metrizable group $K$, a closed subgroup $H\le K$,
      and a homomorphism $\alpha:G\to K$ with dense image such that $(Z,G)$
      is conjugate to the left-translation action on $K/H$.  The system is
      uniquely ergodic, with invariant measure induced by normalized Haar
      measure on $K$.
\item If $p:(Z,G)\to(Y,G)$ is a factor map between minimal equicontinuous
      systems, all fibers of $p$ are homeomorphic and have the same
      cardinality.  The disintegration of the unique invariant probability
      measure of $Z$ over $p$ has full support on every fiber.  If the fibers
      are finite, the conditional measure is uniform on each fiber.
\end{enumerate}
\end{lemma}

\begin{proof}
(a) is the standard compact homogeneous-space representation
cited above.  By taking $K$ to be the closure of the image of $G$ in
$\Isom(Z)$, the homomorphism $\alpha:G\to K$ has dense image.  Unique
ergodicity follows from the uniqueness of the invariant probability measure
on the compact homogeneous space $K/H$.

For (b), let
\[
\rho_Z:G\to\Isom(Z)
\qquad \text{and}\qquad
\rho_Y:G\to\Isom(Y)
\]
be the homomorphisms induced by the two actions, and put
\[
L:=
\overline{
\left\{
\bigl(\rho_Z(g),\rho_Y(g)\bigr):g\in G
\right\}}
\subseteq
\Isom(Z)\times\Isom(Y),
\]
where the closure is taken in the product of the uniform topologies.
The group $L$ acts on $Z$ and $Y$ through the two coordinate
projections:
\[
(\ell_Z,\ell_Y)z:=\ell_Z(z)
\qquad\text{and}\qquad
(\ell_Z,\ell_Y)y:=\ell_Y(y).
\]
The $G$-equivariance and uniform topologies ensure that the factor map $p$ is $L$-equivariant.

Fix $z_0\in Z$ and set $y_0:=p(z_0)$.  Denote the corresponding stabilizers
by
\[
L_{z_0}:=\{\ell\in L:\ell z_0=z_0\},
\qquad\text{and}\qquad
L_{y_0}:=\{\ell\in L:\ell y_0=y_0\}.
\]
Equivariance of $p$ gives
$
L_{z_0}\le L_{y_0},
$
and the orbit maps identify
$
Z\simeq L/L_{z_0}
$ and $Y\simeq L/L_{y_0}$,
under which $p$ becomes the canonical quotient map
$
L/L_{z_0}\to L/L_{y_0}.
$
Consequently,
\begin{equation}\label{eq:homogeneous-orbit-map}
p^{-1}(y_0)\simeq L_{y_0}/L_{z_0},    
\end{equation}
and every other fiber is an $L$-translate of this fiber.  Hence all fibers
are homeomorphic and have the same cardinality.

Let $\lambda_0$ be the push-forward of normalized Haar measure on
$L_{y_0}$ to the homogeneous space
$
L_{y_0}/L_{z_0}.
$
For $y=\ell y_0$, define
$
\lambda_y:=\ell_*\lambda_0.
$
This is independent of the choice of $\ell$, since $\lambda_0$ is
$L_{y_0}$-invariant.  The quotient integral formula for compact groups and
\eqref{eq:homogeneous-orbit-map}
show that
\[
m_Z=\int_Y\lambda_y\,dm_Y(y)\qquad\text{and} \qquad\supp\lambda_y\subset p^{-1}(y),\quad\text{for all }y\in Y.
\]
so $(\lambda_y)_{y\in Y}$ is a version of the disintegration of $m_Z$ over
$p$.  Each $\lambda_y$ has full support because Haar measure has full
support and the relevant quotient maps are surjective.  If the fibers are
finite, the transitive action on each fiber forces $\lambda_y$ to be the
uniform probability measure.
\end{proof}

\subsubsection{Measure-theoretic maximal compact factors}

We now pass from the topological category to the measure-theoretic one.
The functions with relatively compact Koopman orbit generate
a maximal factor with discrete spectrum.  Mackey's representation theorem \cite{Mackey1964}
shows that this factor is again modeled by a compact homogeneous space
$K/H$.  In contrast to the preceding subsection, the resulting
identification and factor map are understood only modulo null sets.
The material in this subsection does not require amenability of the acting
group. 

Let
$
(X,\mu,G)
$
be an ergodic p.m.p.\ $G$-system, and let
$
U_gf(x):=f(g^{-1}x),
$ $
g\in G,$
be its Koopman representation on $L^2(X,\mu)$.  Define the compact
subspace by
\[
\mathcal H_{\mathrm c}(\mu)
:=
\left\{
f\in L^2(X,\mu):
\{U_gf:g\in G\}
\text{ is relatively compact in }L^2(X,\mu)
\right\}.
\]
By \cite{Zimmer1976Extensions}, there is a unique complete $G$-invariant sub-sigma-algebra
$\cK_\mu\subseteq\mc B_X$, modulo $\mu$, such that
\[
L^2(X,\cK_\mu,\mu)=\mathcal H_{\mathrm c}(\mu).
\]
It is called the \emph{maximal compact $\sigma$-algebra} of
$(X,\mu,G)$.  The associated factor is the maximal compact factor of the
system.  In the abelian setting, it is also commonly called the \emph{Kronecker
factor}.

The Koopman representation of the maximal compact factor is a discrete
direct sum of finite-dimensional irreducible representations.  Hence
Mackey's representation theorem
\cite[Theorem~1]{Mackey1964} yields the following compact
homogeneous model.
\begin{theorem}
\label{thm:compact-homogeneous-model}
There exist a compact metrizable group $K_\mu$, a closed subgroup
$H_\mu\le K_\mu$, a homomorphism
$
\alpha_\mu:G\to K_\mu
$
with dense image, and a measure-theoretic factor map
\[
\kappa_\mu:
(X,\mu,G)
\to
(Z_\mu,m_\mu,G),
\qquad
Z_\mu:=K_\mu/H_\mu,
\]
such that
\[
\kappa_\mu^{-1}\bigl(\cB(Z_\mu)\bigr)
=
\cK_\mu
\pmod\mu.
\]
The $G$-action on $Z_\mu$ is given by
$
g(kH_\mu)
=
\alpha_\mu(g)kH_\mu,
$
and $m_\mu$ is the push-forward of normalized Haar measure on
$K_\mu$ under the quotient map
$
K_\mu\to K_\mu/H_\mu.
$
\end{theorem}

The following result was proved in
\cite[Lemma~3.5]{LiuWangXu2025}; see also
\cite[Theorem~16, p.~53]{HostKra2018} for the underlying argument.
\begin{lemma}
\label{lem:mixed-invariant-algebra}
Let $(X,G)$ be a $G$-system, let
$
\mu_1,\ldots,\mu_N\in\M_G^e(X),
$
and put
$
\mu:=\mu_1\otimes\cdots\otimes\mu_N.
$
Then
\[
\mc I_\mu
\subseteq
\cK_{\mu_1}\otimes\cdots\otimes\cK_{\mu_N}
\pmod \mu.
\]
\end{lemma}

\subsubsection{Measurable rigidity between compact models}

The preceding two subsections describe compact systems in the topological
and measure-theoretic categories, respectively.  Although a
measure-theoretic factor map is initially defined only modulo null sets,
we may compare the measure-theoretic
maximal compact factor of an ergodic measure with the topological maximal
equicontinuous factor of the ambient system.

The following measurable rigidity lemma extends the argument for
$\mathbb Z^d$-actions in \cite[Proposition~10]{Cortez2006}.
\begin{lemma}
\label{lem:measurable-factor-continuous}
Let $(Z,G)$ and $(Y,G)$ be minimal equicontinuous systems\footnote{No amenability assumption on $G$ is required.}, and let
$m_Z$ and $m_Y$ denote their unique invariant probability measures.
Suppose that
$
p^0:(Z,m_Z,G)\to(Y,m_Y,G)
$
is a measure-theoretic factor map. Then there exists a continuous
topological factor map
$
p:(Z,G)\to(Y,G)
$
such that
$
p=p^0$,
$m_Z\text{-a.e.}$
\end{lemma}
\begin{proof}
Let
\[
 \rho:=(\id_Z,p^0)_*m_Z
 \qquad\text{and}\qquad
 \Gamma:=\supp(\rho)\subseteq Z\times Y.
\]
The measure $\rho$ is ergodic, and the product action is equicontinuous.  Let
$K$ be the closure of the diagonal image of $G$ in the isometry group of
$Z\times Y$.  Since convergence in this isometry group is uniform,
invariance under $G$ implies invariance of $\rho$ under $K$.  The $K$-orbits
are compact and are precisely the minimal components.  The compact Hausdorff
quotient by the $K$-orbits carries the trivial $G$-action, so the push-forward
of the ergodic measure $\rho$ to this quotient is a Dirac measure.  Hence
$\rho$ is concentrated on one $K$-orbit.  Since $\rho$ is $K$-invariant and
the invariant measure on that orbit has full support, $\Gamma$ is exactly
this orbit and is a minimal equicontinuous subsystem.

The first projection $q:\Gamma\to Z$ is a factor map and $q_*\rho=m_Z$.  By
Lemma \ref{lem:homogeneous-factor-fibers}, the conditional measures of $\rho$ over
$q$ have full support on the fibers of $q$.  On the other hand, since $\rho$
is a graph measure, these conditionals are Dirac measures for
$m_Z$-a.e. base point.  Therefore the fibers of $q$ are singletons
a.e.  Their cardinality is constant by
Lemma \ref{lem:homogeneous-factor-fibers}, so $q$ is one-to-one everywhere and
hence a homeomorphism.  The map
$
 p:=\operatorname{pr}_Y\circ q^{-1}
$
is the required continuous factor and agrees with $p^0$ almost everywhere.
\end{proof}

The next consequence will be used in both the degree and independence arguments below.
\begin{proposition}\label{prop1}
Let $\mu\in\M_G^e(X)$, and let
$
\kappa_\mu:(X,\mu,G)\to(Z_\mu,m_\mu,G)
$
be a homogeneous model of its measure-theoretic maximal compact factor.
Then there is a continuous factor map
$
p_\mu:(Z_\mu,G)\to(\Xeq,G)
$
such that
\[
\pi=p_\mu\circ\kappa_\mu
\qquad \mu\text{-a.e.}
\]
Moreover, all fibers of $p_\mu$ have the same cardinality.
\end{proposition}

\begin{proof}
The topological factor $\pi$ is a compact measure-theoretic factor of
$(X,\mu,G)$, so maximality of $\kappa_\mu$ gives a
measure-theoretic factor map
$p_\mu^0:(Z_\mu,m_\mu,G)\to(\Xeq,\nu,G)$ with
$\pi=p_\mu^0\circ\kappa_\mu$, $\mu$-a.e.
Lemma~\ref{lem:measurable-factor-continuous} replaces $p_\mu^0$ by an
almost everywhere equal continuous factor map $p_\mu$.  The final
assertion follows from
Lemma~\ref{lem:homogeneous-factor-fibers}.
\end{proof}

\begin{corollary}
\label{cor:automatic-phase-persistence}
Let $\mu\in\M_G^e(X)$, let
$
\kappa_\mu:(X,\mu,G)\to(Z_\mu,m_\mu,G)
$
be a homogeneous model of its maximal compact factor, and put
\[
\Gamma_\mu
:=
\supp\bigl((\id_X,\kappa_\mu)_*\mu\bigr),
\qquad
A_\mu(x):=\{z\in Z_\mu:(x,z)\in\Gamma_\mu\}.
\]
Then $A_\mu:X\to\mc H(Z_\mu)$ is  continuous and equivariant.
For the continuous factor $p_\mu:Z_\mu\to\Xeq$ constructed in Proposition \ref{prop1},
\begin{equation}\label{eq:exact-compatible-phases-general}
A_\mu(x)=p_\mu^{-1}(\pi(x))
\qquad(x\in X).
\end{equation}
\end{corollary}

\begin{proof}
The first marginal of the graph joining
$(\id_X,\kappa_\mu)_*\mu$ is $\mu$, which has full support by
minimality.  Hence the first projection of $\Gamma_\mu$ is all of $X$,
and Lemma~\ref{lem:automatic-fiber-continuity} applies.

Since
$
\pi=p_\mu\circ\kappa_\mu
$ $\mu\text{-a.e.},
$
the support of the graph joining is contained in
\[
\{(x,z)\in X\times Z_\mu:\pi(x)=p_\mu(z)\}.
\]
Therefore
\begin{equation}\label{eq:automatic-A-in-fiber}
A_\mu(x)\subseteq p_\mu^{-1}(\pi(x))
\qquad(x\in X).
\end{equation}
The image $A_\mu(X)$ is a minimal equicontinuous factor of $X$ under
the induced isometric action on $\mc H(Z_\mu)$.  By the universal
property of $\Xeq$, there is a continuous map
\[
Q_\mu:\Xeq\to\mc H(Z_\mu)
\]
such that $A_\mu=Q_\mu\circ\pi$.  In view of
\eqref{eq:automatic-A-in-fiber},
\begin{equation}\label{eq:20267261658}
Q_\mu(y)\subseteq p_\mu^{-1}(y)
\qquad(y\in\Xeq).    
\end{equation}

For $\mu$-a.e. $x\in X$,
\[
\kappa_\mu(x)\in A_\mu(x)
=Q_\mu(\pi(x))
=Q_\mu(p_\mu(\kappa_\mu(x))).
\]
Pushing forward by $\kappa_\mu$ and disintegrating
$
m_\mu=\int_{\Xeq}(m_\mu)_y\,d\nu(y)
$
over $p_\mu$, we obtain
\[
(m_\mu)_y(Q_\mu(y))=1
\qquad\text{for $\nu$-a.e. }y\in\Xeq.
\]
The conditional measure $(m_\mu)_y$ has full support on
$p_\mu^{-1}(y)$ by Lemma~\ref{lem:homogeneous-factor-fibers}.  Since
$Q_\mu(y)$ is closed and contained in that fiber, it follows from \eqref{eq:20267261658} that
\[
Q_\mu(y)=p_\mu^{-1}(y)
\qquad\text{for $\nu$-a.e. }y\in\Xeq.
\]

Finally, apply Lemma~\ref{lem:automatic-fiber-continuity} to the closed
invariant relation
\[
\{(y,z)\in\Xeq\times Z_\mu:p_\mu(z)=y\}.
\]
The map $y\mapsto p_\mu^{-1}(y)$ is  continuous.  Both it and
$Q_\mu$ are therefore continuous and agree on a $\nu$-conull, hence
dense, subset of $\Xeq$.  They agree everywhere, proving
\eqref{eq:exact-compatible-phases-general}.
\end{proof}
\begin{remark}\label{rem:model-independence-multiplicities}
Corollary~\ref{cor:automatic-phase-persistence} shows that
\[
 \iota_\mu
 :=|p_\mu^{-1}(y)|
 =|A_\mu(x)|
 \in\N\cup\{\infty\}
\]
is independent of $x\in X$ and $y\in\Xeq$.  It is also independent of
the chosen compact homogeneous model: any two models represent the same
maximal compact factor and are measurably conjugate over $\Xeq$, while
Lemma~\ref{lem:measurable-factor-continuous} upgrades the induced map
between their compact models to a continuous factor isomorphism.  The
residual multiplicity $b_\mu$ is model-independent as well, since
$b_\mu=\exp(h_\mu^*(G))$ by
Proposition~\ref{prop:residual-sequence-entropy}.
\end{remark}
\subsection{Topological entropy for amenable group actions}
For a finite open
cover $\mathcal U$ of $X$,
we write $N(\mathcal U)$ for the least cardinality of a subcover of
$\mathcal U$.  If $\Fcal=(F_n)_{n\ge1}$ is a left F{\o}lner sequence,
define
\[
h_{\mathrm{top}}(X,G,\mathcal U)
:=
\lim_{n\to\infty}
\frac{1}{|F_n|}
\log N\left(\bigvee_{g\in F_n}g^{-1}\mathcal U\right),
\]
and
\[
h_{\mathrm{top}}(X,G)
:=
\sup_{\mathcal U}
h_{\mathrm{top}}(X,G,\mathcal U),
\]
where the supremum is taken over all finite open covers of $X$.
The above limit exists and is independent of the choice of the
F{\o}lner sequence; see
\cite[Lemma~2.4 and Section~3.1]{HuangYeZhang2011}.
\subsection{Sequence entropy and independence}
Fix $\mu\in\M_G(X)$.  For a finite measurable partition
$\alpha$ of $X$, define its \emph{Shannon entropy} by
\[
H_\mu(\alpha)
:=
-\sum_{A\in\alpha}\mu(A)\log\mu(A).
\]
 If $\alpha$ and $\beta$ are
finite measurable partitions, their join is
$
 \alpha\vee\beta
 :=
 \{A\cap B:A\in\alpha,\ B\in\beta\},
$
after sets of $\mu$-measure zero are discarded.

Let
$
\mathsf S=(g_n)_{n\ge1}
$
be a sequence of pairwise distinct elements of $G$, and let $\alpha$
be a finite measurable partition of $X$.  Define
\[
h_\mu^{\mathsf S}(G,\alpha)
:=
\limsup_{n\to\infty}
\frac1n
H_\mu
\left(
\bigvee_{j=1}^n g_j^{-1}\alpha
\right),
\]
and
\[
h_\mu^*(G,\alpha)
:=
\sup_{\mathsf S}
h_\mu^{\mathsf S}(G,\alpha),
\qquad
h_\mu^*(G)
:=
\sup_{\alpha}
h_\mu^*(G,\alpha).
\]
The following result is a direct specialization of
\cite[Theorem~2.9]{LiuWangXu2025}. For $\mathbb Z$-actions, the
underlying sequence-entropy theory goes back to Kushnirenko
\cite{Kushnirenko1967} and Hulse \cite{Hulse1982}; the partition-level
conditional-entropy formula was established by Huang--Maass--Ye
\cite[Lemma~2.2 and Theorem~2.3]{HuangMaassYe2004}.
An amenable-group counterpart was obtained in
\cite{LiuYan2023}.

\begin{proposition}
\label{prop:residual-sequence-entropy}
Let
$
\mu\in\M_G^e(X),
$
and let
$
\kappa_\mu:
(X,\mu,G)
\to
(Z_\mu,m_\mu,G)
$
be its maximal compact factor.  Write
$
\mu
=
\int_{Z_\mu}\mu_{\mu,z}\,dm_\mu(z)
$
for the disintegration of $\mu$ over $\kappa_\mu$.  Then there exists
$
b_\mu\in\N\cup\{\infty\}
$
such that
\[
|\supp(\mu_{\mu,z})|=b_\mu
\qquad
\text{for }m_\mu\text{-a.e. }z\in Z_\mu,
\]
and
\[
h_\mu^*(G)=\log b_\mu.
\]
\end{proposition}
We also record the topological counterpart. Let
$\mathsf S=(g_n)_{n\ge1}$ be a sequence of pairwise distinct elements of
$G$, and let $\mathcal U$ be a finite open cover of $X$. Set
\[
h_{\mathrm{top}}^{\mathsf S}(G,\mathcal U)
:=
\limsup_{n\to\infty}\frac1n
\log N\left(\bigvee_{j=1}^n g_j^{-1}\mathcal U\right).
\]
Define
\[
h_{\mathrm{top}}^{\mathsf S}(X,G)
:=
\sup_{\mathcal U}
h_{\mathrm{top}}^{\mathsf S}(G,\mathcal U),
\qquad
h_{\mathrm{top}}^*(X,G)
:=
\sup_{\mathsf S}
h_{\mathrm{top}}^{\mathsf S}(X,G).
\]
The quantity $h_{\mathrm{top}}^*(X,G)$ is the maximal topological
sequence entropy, also called maximal pattern entropy; see
\cite{Goodman1974,HuangYe2009,LiuYan2023}.

We next recall the independence notions \cite{KerrLi2007}.  Let $N\ge2$ and let
$A_1,\ldots,A_N\subseteq X$.  A set $J\subseteq G$ is called an
\emph{independence set} for $(A_1,\ldots,A_N)$ if
$
\bigcap_{g\in F}g^{-1}A_{\omega(g)}\ne\emptyset
$
for every nonempty finite set $F\subseteq J$ and every map
$
\omega:F\to[N].
$

\begin{definition}
\label{def:IN-IT-tuples}
Let $N\ge2$ and let
$
\mathbf x=(x_1,\ldots,x_N)\in X^N.
$
We call $\mathbf x$ 
\begin{enumerate}
    \item an \emph{IN $N$-tuple} if, for every choice of
neighborhoods $U_i\ni x_i$, $i\in[N]$, there exist independence sets
for $(U_1,\ldots,U_N)$ of arbitrarily large finite cardinality;
\item an \emph{IT $N$-tuple} if, for every choice of
neighborhoods $U_i\ni x_i$, $i\in[N]$, there exists an infinite
independence set for $(U_1,\ldots,U_N)$.
\end{enumerate}
We denote the sets of IN and IT $N$-tuples by
$
\IN_N(X,G)
$ and  $
\IT_N(X,G),$
respectively.
\end{definition}

The notion of an IN-tuple is combinatorial, whereas a sequence
entropy tuple is defined in terms of positive topological sequence
entropy along a suitable sequence of group elements.  These notions
are nevertheless equivalent off the fat diagonal: for every
$N\ge2$, the essential IN $N$-tuples are precisely the sequence
entropy $N$-tuples; see \cite{KerrLi2007,HuangYe2009}.  Thus IN-tuples
provide the independence-theoretic description of local sequence
entropy.  By contrast, the IT condition is stronger, since it requires
a single infinite independence set rather than arbitrarily large
finite ones.  In particular,
\[
\IT_N(X,G)\subseteq\IN_N(X,G).
\]
We shall use the following tuple characterization of maximal
topological sequence entropy:
\begin{equation}
\label{eq:maximal-pattern-IN-rank}
h_{\mathrm{top}}^*(X,G)
=
\log\sup\Bigl(
\{1\}\cup
\bigl\{
N\ge2:
\IN_N(X,G)\setminus\Delta_{\mathrm{fat}}^N(X)
\ne\emptyset
\bigr\}
\Bigr),
\end{equation}
which is obtained by combining \cite[Theorem~5.9]{KerrLi2007} and \cite[Theorems~A.1 and A.3]{HuangYe2009}.

\subsection{Multivariate mean equicontinuity}

Let $\mc F=(F_n)_{n\ge1 }$ be a F{\o}lner sequence.
For $r\ge2$ and $\mathbf x=(x_1,\ldots,x_r)\in X^r$, put
\[
 D_r(\mathbf x):=\min_{1\le i<j\le r}d(x_i,x_j)
\]
and
\begin{equation}\label{eq:Folner-mean-distance}
 D_{r,\Fcal}(\mathbf x)
 :=
 \limsup_{n\to\infty}
 \frac1{|F_n|}
 \sum_{g\in F_n}D_r(gx_1,\ldots,gx_r).
\end{equation}
The Weyl mean $r$-distance is
\begin{equation}\label{eq:Weyl-distance}
 D_r^W(\mathbf x):=\sup_{\Fcal}D_{r,\Fcal}(\mathbf x),
\end{equation}
where the supremum is over all  F{\o}lner sequences of $G$.

\begin{definition}\label{def:mean-equicontinuity}
Let $r\ge2$.
\begin{enumerate}[label=\textup{(\alph*)},leftmargin=2.2em]
\item The system is \emph{$\Fcal$-mean $r$-equicontinuous} if, for every
$\varepsilon>0$, there is $\delta>0$ such that
\[
 \max_{i<j}d(x_i,x_j)<\delta
 \quad\Longrightarrow\quad
 D_{r,\Fcal}(x_1,\ldots,x_r)<\varepsilon.
\]
\item It is \emph{Weyl mean $r$-equicontinuous} if the same statement holds with
$D_r^W$ in place of $D_{r,\Fcal}$.
\item A point $x\in X$ is an $\Fcal$-mean $r$-sensitive point if there is
$\varepsilon>0$ such that every neighborhood $U$ of $x$ contains
$x_1,\ldots,x_r$ satisfying
$
 D_{r,\Fcal}(x_1,\ldots,x_r)\ge\varepsilon.
$
\end{enumerate}
\end{definition}
\begin{remark}\label{lem:metric-independence}
    The notions of $\Fcal$-mean and Weyl mean $r$-equicontinuity do not depend
on the chosen compatible metric on the compact space $X$.
\end{remark}
The following lemma is immediate from the definition.
\begin{lemma}\label{lem:no-sensitive-point}
If $(X,G)$ is $\Fcal$-mean $r$-equicontinuous, then it has no
$\Fcal$-mean $r$-sensitive point.
\end{lemma}

\subsection{Conditional topomorphic degree}

Assume that $(X,G)$ is minimal, and let
$
 \pi:(X,G)\to(X_{\mr{eq}},G)
$
be its maximal equicontinuous factor.  A minimal equicontinuous system is
uniquely ergodic; its invariant probability measure will be denoted by
$\nu$. 
For $\mu\in\M_G(X)$, write
$
 \mu=\int_{\Xeq}\mu_y\,d\nu(y)
$
for its disintegration over $\pi$.

\begin{definition}
\label{def:topomorphic-degree}
The \emph{conditional topomorphic degree} of $(X,G)$ over its maximal
equicontinuous factor is
\begin{equation}\label{eq:topomorphic-degree}
 \tdeg_G(X)
 :=
 \sup_{\mu\in\M_G(X)}
 \esssup_{y\in \Xeq}|\supp(\mu_y)|
 \in\N\cup\{\infty\}.
\end{equation}
\end{definition}
\begin{remark}
    The degree $\tdeg_G(X)$ is independent of
the chosen version of the disintegration, since two versions agree for
$\nu$-a.e. base point.  It is also invariant under topological conjugacy.
\end{remark}

\begin{proposition}
\label{prop:finite-conditional-degree}
Assume that $d:=\tdeg_G(X)<\infty$.  Then there are finitely many ergodic
measures $\mu_1,\ldots,\mu_\ell$ and integers $k_i\ge1$ such that
\begin{equation}\label{eq:sum-conditional-degrees}
 \ell\le d,
 \qquad
 \sum_{i=1}^{\ell}k_i=d,
\end{equation}
and the conditional measures of $\mu_i$ over $\pi$ have support cardinality
$k_i$ for $\nu$-a.e. $y$.  Moreover, there are Borel maps
$
 \gamma_{i,j}:\Xeq\to X,$ 
 $
 1\le i\le\ell,
 $ $ 1\le j\le k_i,
$
such that
\[
 \pi(\gamma_{i,j}(y))=y\qquad\text{for every }y\in\Xeq,
\]
 and
\[
 \supp((\mu_i)_y)
 =
 \{\gamma_{i,1}(y),\ldots,\gamma_{i,k_i}(y)\}\qquad\text{for }\nu\text{-a.e. } y\in \Xeq.
\]
\end{proposition}

\begin{proof}
If $\mu\in\M_G^e(X)$, as the function $y\mapsto |\supp\mu_y|$  is  a
$G$-invariant measurable function of $y$ and hence is constant,  $\nu$-a.e.; denote its value by $k_\mu$.  By the definition of $d$, one has
$1\le k_\mu\le d$.

Let $\mu_1,\ldots,\mu_s$ be distinct ergodic measures and choose pairwise
disjoint invariant Borel sets $E_i$ with $\mu_i(E_i)=1$.  On a common conull
subset of $\Xeq$, each finite atomic measure $(\mu_i)_y$ is concentrated on
$E_i$.  Every point of its finite support is an atom of positive mass and
therefore belongs to $E_i$.  Thus the conditional supports for different
$i$ are disjoint.  On the same common conull set, the conditional measure of
$
 \frac1s\sum_{i=1}^s\mu_i
$
over $\pi$ is
$
 \frac1s\sum_{i=1}^s(\mu_i)_y
$
and has exactly $\sum_i k_{\mu_i}\ge s$ support points.  Hence $s\le d$.
Consequently there are only finitely many ergodic measures, say
$\mu_1,\ldots,\mu_\ell$.

The average $\ell^{-1}\sum_i\mu_i$ shows that
$\sum_i k_i\le d$.  Conversely, every invariant measure is a convex
combination of the finite family $\mu_i$, and its conditional support is
contained in the union of their conditional supports.  Hence
$d\le\sum_i k_i$, proving \eqref{eq:sum-conditional-degrees}.

For fixed $i$, choose a Borel version of the conditional kernel and
define the support relation
\[
 S_i:=\{(y,x)\in \Xeq\times X:x\in\supp((\mu_i)_y)\}.
\]
This relation is Borel.  Indeed, for a countable base $\mathcal U$ of $X$,
\[
 (\Xeq\times X)\setminus S_i
 =
 \bigcup_{U\in\mathcal U}
 \bigl(\{y:(\mu_i)_y(U)=0\}\times U\bigr).
\]
Let $Y_i\subseteq\Xeq$ be a conull Borel set such that
\[
|S_i(y)|=k_i
\qquad
\text{for every }y\in Y_i.
\]
Applying the Lusin--Novikov uniformization theorem
\cite[Theorem~18.10]{Kechris} to
$
S_i\cap(Y_i\times X),
$
we obtain Borel maps
$
\gamma_{i,1}^0,\ldots,\gamma_{i,k_i}^0:Y_i\to X
$
such that, for every $y\in Y_i$, the points
$
\gamma_{i,1}^0(y),\ldots,\gamma_{i,k_i}^0(y)
$
are pairwise distinct and
\[
S_i(y)
=
\{\gamma_{i,1}^0(y),\ldots,\gamma_{i,k_i}^0(y)\}.
\]
In particular,
\[
\pi(\gamma_{i,j}^0(y))=y
\qquad
(y\in Y_i,\ 1\le j\le k_i).
\]

To extend these maps to all of $\Xeq$, consider the closed relation
\[
R
:=
\{(y,x)\in\Xeq\times X:\pi(x)=y\}.
\]
Since $\pi:X\to\Xeq$ is a continuous surjection between compact
metric spaces, every vertical section of $R$ is nonempty and compact.
The Arsenin--Kunugui uniformization theorem
\cite[Theorem~18.18]{Kechris} therefore yields a Borel map
$
s:\Xeq\to X
$
such that
\[
\pi(s(y))=y
\qquad
\text{for every }y\in\Xeq.
\]
For $j=1,\ldots,k_i$, define
\[
\gamma_{i,j}(y)
:=
\begin{cases}
\gamma_{i,j}^0(y),&y\in Y_i,\\
s(y),&y\notin Y_i.
\end{cases}
\]
Then $\gamma_{i,j}:\Xeq\to X$ is Borel,
$
\pi\circ\gamma_{i,j}=\id_{\Xeq},
$
and
\[
\supp((\mu_i)_y)
=
\{\gamma_{i,1}(y),\ldots,\gamma_{i,k_i}(y)\}\qquad\text{for $\nu$-a.e. $y\in\Xeq$.}\qedhere
\]

\end{proof}

\section{Mixed compact-phase joinings and localization}
\label{sec:compact-phase-decompositions}

The purpose of this section is to convert information from the maximal
compact factors into local orbit configurations in a minimal $G$-system $(X,G)$.  Given ergodic
measures $\mu_1,\ldots,\mu_N$, we first decompose the product measure on
their compact factors into ergodic joint-phase orbit measures
$\lambda_t$.  We then lift these measures, through the disintegrations
over the compact factors, to mixed joinings $\Lambda_t$ on $X^N$.
We prove that almost every $\Lambda_t$ is ergodic and that the mass of
every Borel rectangle depends continuously on the phase parameter.

These results yield the local realization theorem (Theorem \ref{thm:mixed-phase-local-realization}).  Suppose that a phase tuple is compatible with
prescribed points $x_1,\ldots,x_N$, and that the corresponding mixed
joining assigns positive mass to a target rectangle
$C_1\times\cdots\times C_N$.  Then, inside arbitrary neighborhoods of
the points $x_q$, one can choose a new tuple whose coordinates visit the
corresponding targets simultaneously with positive upper density along any
prescribed   F{\o}lner sequence.  In particular, positive off-diagonal
mass produces a  mean-sensitive tuple.  We apply this
principle in two complementary ways: first to bound the number of distinct
measure--phase labels carried by a point, and then to bound the residual
support multiplicity over a fixed compact phase.

\subsection{Phase orbits in the product compact model}
\label{subsec:compact-phase-orbits}
For $q\in[N]$,  let $\mu_q\in\mc M_G^e(X)$, and choose a maximal compact
factor model
\[
\kappa_q:
(X,\mu_q,G)
\to
(Z_q,m_q,G),
\qquad
Z_q:=K_q/H_q,
\]
as in Theorem~\ref{thm:compact-homogeneous-model}.  

Let $m_{K_q}$ be normalized Haar measure on $K_q$, and let
$
\tau_q:K_q\to Z_q
$
be the quotient map.  Thus
$
m_q:=(\tau_q)_*m_{K_q}.
$

Put
\[
\mathbf Z:=Z_1\times\cdots\times Z_N,
\quad
\mathbf m:=m_1\otimes\cdots\otimes m_N,
\quad\text{and}\quad
P:=K_1\times\cdots\times K_N.
\]
Define
\[
\boldsymbol\alpha:G\to P,
\qquad
\boldsymbol\alpha(g):=
\bigl(\alpha_1(g),\ldots,\alpha_N(g)\bigr),
\]
and let
\begin{equation}\label{eq:phase-groups}
L:=\overline{\boldsymbol\alpha(G)}\le P.
\end{equation}
We write
$
m_P:=m_{K_1}\otimes\cdots\otimes m_{K_N}
$
and denote normalized Haar measure on $L$ by $m_L$.

Define the quotient map
\begin{equation}\label{eq:quotient map}
\tau:P\to\mathbf Z,
\qquad
\tau(t_1,\ldots,t_N)
:=
(t_1H_1,\ldots,t_NH_N).    
\end{equation}
For $t\in P$, define
\begin{equation}\label{eq:lambda-phase}
\lambda_t
:=
\int_L\delta_{\tau(\ell t)}\,dm_L(\ell),
\end{equation}
where multiplication in $\ell t$ is coordinatewise.
\begin{lemma}
\label{lem:compact-phase-orbit-measures}
The family $t\mapsto\lambda_t$ has the following properties.
\begin{enumerate}[label=\textup{(\roman*)},leftmargin=2.2em]
\item
It is continuous for the weak* topology.

\item
For every $t\in P$, the measure $\lambda_t$ is $G$-invariant and
ergodic.

\item
For every $q$, the $q$-th marginal of $\lambda_t$ is $m_q$.

\item 
$
\int_P\lambda_t\,dm_P(t)=\mathbf m.
$
\end{enumerate}
\end{lemma}

\begin{proof}
For $F\in C(\mathbf Z)$,
\[
\int_{\mathbf Z}F\,d\lambda_t
=
\int_LF\bigl(\tau(\ell t)\bigr)\,dm_L(\ell).
\]
The map
$
(t,\ell)\mapsto F\bigl(\tau(\ell t)\bigr)
$
is continuous on the compact space $P\times L$.  It follows that the
right-hand side depends continuously on $t$, proving
\textup{(i)}.

For $g\in G$, left invariance of $m_L$ gives
\[
g_*\lambda_t=
\int_L
\delta_{\tau(\boldsymbol\alpha(g)\ell t)}\,dm_L(\ell)=
\int_L
\delta_{\tau(\ell t)}\,dm_L(\ell)
=
\lambda_t.
\]
Thus $\lambda_t$ is $G$-invariant.

Let
$
\mathcal O_t:=\tau(Lt).
$
Since $\boldsymbol\alpha(G)$ is dense in $L$, every
$G$-invariant probability measure on $\mathcal O_t$ is
$L$-invariant.  The compact group $L$-action on $\mathcal O_t$ is transitive, so
$\mathcal O_t$ carries a unique $L$-invariant probability measure,
namely $\lambda_t$.  Hence $\lambda_t$ is ergodic.  This proves
\textup{(ii)}.

For each $q\in[N]$, the coordinate projection
$
\operatorname{pr}_q:L\to K_q
$
has compact image containing the dense subgroup $\alpha_q(G)$, and
hence is surjective.  Therefore
\[
(\operatorname{pr}_q)_*m_L=m_{K_q}.
\]
Using the right invariance of Haar measure, the $q$-th marginal of
$\lambda_t$ is
\[
\bigl(k\mapsto kt_qH_q\bigr)_*m_{K_q}
=
(\tau_q)_*m_{K_q}
=
m_q.
\]
This proves \textup{(iii)}.

Finally, for $F\in C(\mathbf Z)$, Fubini's theorem and left invariance of
$m_P$ give
\begin{align*}
\int_P\int_{\mathbf Z}F\,d\lambda_t\,dm_P(t)
=
\int_L\int_P
F\bigl(\tau(\ell t)\bigr)\,dm_P(t)\,dm_L(\ell)=
\int_PF\bigl(\tau(t)\bigr)\,dm_P(t)=
\int_{\mathbf Z}F\,d\mathbf m.
\end{align*}
This proves \textup{(iv)}.
\end{proof}

\subsection{Lifting phase measures to the original systems}
\label{subsec:lifted-phase-joinings}

For each $q\in[N]$, choose a Borel disintegration
$
\mu_q
=
\int_{Z_q}\mu_{q,z}\,dm_q(z)
$
over $\kappa_q$.  Since $G$ is countable, the kernels may be chosen
so that there exists a $G$-invariant conull Borel set
$Z_q^0\subseteq Z_q$ on which
\[
g_*\mu_{q,z}=\mu_{q,gz}
\qquad
(g\in G)
\qquad\text{and}\qquad
\mu_{q,z}\bigl(\kappa_q^{-1}(\{z\})\bigr)=1.
\]

Put
\[
\mathbf X:=X^N
\qquad\text{and}\qquad
\boldsymbol\mu:=\mu_1\otimes\cdots\otimes\mu_N.
\]
For
$
z=(z_1,\ldots,z_N)\in\mathbf Z,
$
define
\[
\xi_z
:=
\mu_{1,z_1}\otimes\cdots\otimes\mu_{N,z_N}.
\]
Then $z\mapsto\xi_z$ is a Borel probability kernel. Define the Borel lifting map
\[
\mathcal T:
\Prob(\mathbf Z)\to\Prob(\mathbf X),
\qquad
\mathcal T(\beta)
:=
\int_{\mathbf Z}\xi_z\,d\beta(z).
\]
Then 
\begin{equation}\label{eq:product-compact-disintegration}
\boldsymbol\mu
=
\int_{\mathbf Z}\xi_z\,d\mathbf m(z)=\mc T(\mathbf m).
\end{equation}

For $t\in P$, define
\begin{equation}\label{eq:Lambda-phase}
\Lambda_t
:=
\mathcal T(\lambda_t)=\int_L\mu_{1,\ell_1t_1H_1}\otimes \cdots\otimes\mu_{N,\ell_Nt_NH_N}\ dm_L(\ell).
\end{equation}
Since the $q$-th marginal of every $\lambda_t$ is $m_q$, the
definition of $\Lambda_t$ is unaffected by changes to the conditional
kernels on $m_q$-null sets.

\begin{proposition}
\label{prop:lifted-phase-joinings}
The map
$
t\mapsto\Lambda_t
$
is Borel.  For every $t\in P$, the measure $\Lambda_t$ is a
$G$-invariant joining of
$
\mu_1,\ldots,\mu_N.
$
Moreover,
\begin{equation}\label{eq:phase-barycenter}
\int_P\Lambda_t\,dm_P(t)
=
\boldsymbol\mu.
\end{equation}
\end{proposition}

\begin{proof}
Since the map
$\mc T$
is Borel on $\Prob(\mathbf Z)$ and 
$t\mapsto\lambda_t$ is Borel, it follows that
\[
t\mapsto
\Lambda_t(B)
=
\mathcal T(\lambda_t)(B)
\]
is Borel.  The evaluation maps
$
\eta\mapsto\eta(B),$ $
B\in\mathcal B(\mathbf X)$
generate the Borel structure of $\Prob(\mathbf X)$.  Therefore
$t\mapsto\Lambda_t$ is Borel.

For a Borel set $A\subseteq X$,
\begin{align*}
(\operatorname{pr}_q)_*\Lambda_t(A)
=
\int_{\mathbf Z}\mu_{q,z_q}(A)\,d\lambda_t(z)=
\int_{Z_q}\mu_{q,z_q}(A)\,dm_q(z_q)=
\mu_q(A).
\end{align*}
Thus $\Lambda_t$ is a joining of the prescribed marginals.
Using the $G$-invariance of $\lambda_t$, we obtain
\[
g_*\Lambda_t
=
\int_{\mathbf Z}g_*\xi_z\,d\lambda_t(z)
=
\int_{\mathbf Z}\xi_{gz}\,d\lambda_t(z)
=
\int_{\mathbf Z}\xi_z\,d\lambda_t(z)
=
\Lambda_t.
\]

Finally, by Lemma~\ref{lem:compact-phase-orbit-measures},
\eqref{eq:product-compact-disintegration}, and Fubini's theorem,
\[
\int_P\Lambda_t\,dm_P(t)
=
\int_P\int_{\mathbf Z}\xi_z\,d\lambda_t(z)\,dm_P(t)=
\int_{\mathbf Z}\xi_z\,d\mathbf m(z)=
\boldsymbol\mu.\qedhere
\]
\end{proof}

We next prove that almost every lifted phase joining $
\Lambda_t=\mathcal T(\lambda_t)
$ is ergodic.
\begin{proposition}
\label{prop:mixed-phase-decomposition}
The joining $\Lambda_t$ is ergodic for
$m_P$-a.e. $t\in P$.
\end{proposition}

\begin{proof}
Let
$
\kappa
:=
\kappa_1\times\cdots\times\kappa_N:
\mathbf X\to\mathbf Z.
$
By Lemma~\ref{lem:mixed-invariant-algebra},
\begin{equation}\label{eq:mixed-invariant-factor}
\mathcal I_{\boldsymbol\mu}
\subseteq
\kappa^{-1}\mathcal B(\mathbf Z)
\qquad
\pmod{\boldsymbol\mu}.
\end{equation}
Choose a standard Borel model $\Omega$ for the invariant sigma-algebra
$\mathcal I_{\boldsymbol\mu}$, together with an invariant component map
$
r:\mathbf X\to\Omega
$
and an ergodic decomposition
$
\boldsymbol\mu
=
\int_\Omega\eta_\omega\,d\vartheta(\omega),
$
where
$
r_*\boldsymbol\mu=\vartheta
$
and $\eta_\omega$ is the conditional measure over $r$.

By \eqref{eq:mixed-invariant-factor}, the map $r$ is measurable with
respect to the factor $\kappa^{-1}\mc B(\mathbf Z)$.  Therefore
there exists a Borel map
$
\sigma:\mathbf Z\to\Omega
$
such that
\begin{equation}\label{eq:component-factorization}
r=\sigma\circ\kappa
\qquad
\boldsymbol\mu\text{-a.e}.
\end{equation}

For $\vartheta$-a.e. $\omega\in\Omega$, put
$
\beta_\omega:=\kappa_*\eta_\omega.
$
Then $\beta_\omega$ is ergodic for $\vartheta$-a.e. $\omega\in\Omega$, and 
\[
\mathbf m
=
\kappa_*\boldsymbol\mu
=
\int_\Omega\beta_\omega\,d\vartheta(\omega).
\]
Consequently,
$
Q_\beta
:=
(\omega\mapsto\beta_\omega)_*\vartheta
$
is the ergodic-decomposition measure of $\mathbf m$.
Moreover, by \eqref{eq:component-factorization} and the fact that
$\eta_\omega$ is concentrated on the fiber $r^{-1}(\{\omega\})$,
the measure $\beta_\omega$ is concentrated on
$\sigma^{-1}(\{\omega\})$.  Thus $(\beta_\omega)_{\omega\in\Omega}$
is the disintegration of $\mathbf m$ over $\sigma$.
By the uniqueness of measure disintegration applied to
\eqref{eq:product-compact-disintegration}, we obtain
\begin{equation}\label{eq:tower-identity}
\eta_\omega
=
\mathcal T(\beta_\omega)
\qquad
\text{for $\vartheta$-a.e. $\omega\in\Omega$.}
\end{equation}

Note that Lemma~\ref{lem:compact-phase-orbit-measures} shows that
$
Q_\lambda
:=
(t\mapsto\lambda_t)_*m_P
$
is another probability measure on $\mathcal M_G^e(\mathbf Z)$ with
barycenter $\mathbf m$.  By uniqueness of the ergodic decomposition,
\begin{equation}\label{eq:ergodic-decomposition-laws}
Q_\beta=Q_\lambda.
\end{equation}
In particular, by \eqref{eq:tower-identity},
\[
Q_\lambda\bigl(
\mathcal T^{-1}(\M_G^e(\mathbf X))
\bigr)
=Q_\beta\bigl(
\mathcal T^{-1}(\M_G^e(\mathbf X))
\bigr)
=
1.
\]
Since
$
\mathcal T(\lambda_t)=\Lambda_t,
$
it follows that
\[
m_P
\bigl(
\{t\in P:\Lambda_t\text{ is ergodic}\}
\bigr)
=
1.\qedhere
\]
\end{proof}

The local realization argument below requires continuity not only against
continuous test functions, but also for the masses assigned by
$\Lambda_t$ to arbitrary Borel rectangles.  The key input is the strong
continuity of right translations on $L^1$ of a compact group.

\begin{lemma}
\label{lem:rectangle-continuity}
Let $C_q\subseteq X$, $q\in[N]$, be Borel sets.  Then
\[
\Phi_C(t)
:=
\Lambda_t(C_1\times\cdots\times C_N)
\]
is continuous on $P$.  Moreover, there exists a unique continuous function
$
\widetilde\Phi_C:\mathbf Z\to[0,1]
$
such that
\[
\Phi_C=\widetilde\Phi_C\circ\tau,
\]
where $\tau$ is the quotient map defined as in \eqref{eq:quotient map}.
\end{lemma}

\begin{proof}
For each $q\in[N]$, define
\[
\varphi_q(k)
:=
\mu_{q,kH_q}(C_q),
\qquad
k\in K_q.
\]
Then $\varphi_q:K_q\to[0,1]$ is Borel and right $H_q$-invariant.
By the definition of $\Lambda_t$,
\[
\Phi_C(t)
=
\int_L
\prod_{q=1}^N
\varphi_q(\ell_qt_q)\,dm_L(\ell).
\]

Let
$
s=(s_1,\ldots,s_N)\in P.
$
Since $0\le\varphi_q\le1$, one has
\begin{align*}
|\Phi_C(ts)-\Phi_C(t)|
&\le
\sum_{q=1}^N
\int_L
\left|
\varphi_q(\ell_qt_qs_q)
-
\varphi_q(\ell_qt_q)
\right|
dm_L(\ell)\\
&\le \sum_{q=1}^N\int_{K_q}
\left|
\varphi_q(kt_qs_q)-\varphi_q(kt_q)
\right|
\,dm_{K_q}(k)\\
&=\sum_{q=1}^N \|R_{s_q}\varphi_q-\varphi_q\|_{L^1(K_q)},
\end{align*}
where
$
R_s\varphi(u):=\varphi(us).
$
Right translations are strongly continuous on $L^1(K_q)$, so the
right-hand side tends to zero as $s\to e_P$.  This proves continuity of
$\Phi_C$.

Let
$
H:=H_1\times\cdots\times H_N.
$
For $h=(h_1,\ldots,h_N)\in H$, one has
$
\Phi_C(th)=\Phi_C(t).
$
Since $\tau:P\to\mathbf Z$ is a quotient map and $\Phi_C$ is continuous, there is a
unique continuous function
$
\widetilde\Phi_C:\mathbf Z\to[0,1]
$
such that
\[
\Phi_C=\widetilde\Phi_C\circ\tau.\qedhere
\]
\end{proof}

\subsection{From phase joinings to local mean sensitivity}\label{sec:localization}
In this subsection, we use the mixed phase joinings constructed above to
prove a local realization theorem.  It converts positive mass of a target
rectangle into positive-density simultaneous visits by a tuple chosen in
arbitrary prescribed source neighborhoods.  The corresponding local mean
sensitivity criterion will be derived afterward.

For $\mu_1,\ldots,\mu_N\in\M_G^e(X)$,  let
\[
 \rho_q:=(\id_X,\kappa_q)_*\mu_q,
 \qquad
 \Gamma_q:=\supp(\rho_q)\subseteq X\times Z_q,\qquad q\in[N]
\]
and define the compact nonempty label set
\begin{equation}\label{eq:label-set}
 A_q(x):=A_{\mu_q}(x)=\{z\in Z_q:(x,z)\in\Gamma_q\},\qquad q\in[N].
\end{equation}
Since $(X,G)$ is minimal, every invariant probability measure has full
support.  Hence
\[
\operatorname{pr}_X(\Gamma_q)
=
\supp\bigl((\operatorname{pr}_X)_*\rho_q\bigr)
=
\supp(\mu_q)
=
X.
\]
Consequently, $A_q(x)$ is nonempty and compact for every $x\in X$.

\begin{theorem}
\label{thm:mixed-phase-local-realization}
Let
$
\mu_1,\ldots,\mu_N\in\mathcal M_G^e(X),
$
and let $\Lambda_t$, $t\in P:=K_1\times\cdots\times K_N$, be the
associated mixed compact-phase joinings.  For each $q=1,\ldots,N$,
fix a point $x_q\in X$ and 
$
z_q^0=t_q^0H_q\in A_q(x_q),
$
and put
$
t^0=(t_1^0,\ldots,t_N^0).
$
Let $C_1,\ldots,C_N\subseteq X$ be Borel sets such that
\[
a:=\Lambda_{t^0}(C_1\times\cdots\times C_N)>0.
\]
Then, for every  F{\o}lner sequence
$
\Fcal=(F_n)_{n\ge1},
$
every choice of neighborhoods $U_q\ni x_q$, and every
$0<\theta<a$, there exist points $y_q\in U_q$,
$q\in[N]$, such that
\[
\limsup_{n\to\infty}
\frac{
\left|
\left\{
g\in F_n:
gy_q\in C_q
\text{ for every }q=1,\ldots,N
\right\}
\right|
}{|F_n|}
>
\theta.
\]
\end{theorem}

\begin{proof}
Fix a  F{\o}lner sequence
$
\Fcal=(F_n)_{n\ge1},
$
neighborhoods $U_q\ni x_q$, and a number $0<\theta<a$.
By Lemma~\ref{lem:rectangle-continuity}, the function
$
(t_1,\ldots,t_N)
\mapsto
\Lambda_t(C_1\times\cdots\times C_N)
$
is a continuous function on
$Z_1\times\cdots\times Z_N$.  Hence there are open neighborhoods
$
W_q\subseteq Z_q$
 of 
$z_q^0=t_q^0H_q,$
such that
\begin{equation}\label{eq:local-realization-target-stability}
\Lambda_t(C_1\times\cdots\times C_N)>\theta
\end{equation}
whenever
$
t_qH_q\in W_q, $ $q\in[N].$

Since $z_q^0\in A_q(x_q)$, one has
$
(x_q,z_q^0)\in\Gamma_q.
$
As $U_q\times W_q$ is an open neighborhood of
$(x_q,z_q^0)$, the definition of $\Gamma_q$ gives
\begin{equation}\label{eq:local-realization-source-mass}
\mu_q\bigl(U_q\cap\kappa_q^{-1}(W_q)\bigr)>0.
\end{equation}

Put
\[
\widetilde W_q:=\tau_q^{-1}(W_q)\subseteq K_q
\]
and define
\[
\psi_q(k):=\mu_{q,kH_q}(U_q),
\qquad
k\in K_q.
\]
For $a_q\in K_q$, set
\[
J_q(a_q)
:=
\int_{\widetilde W_q}
\psi_q(a_qu)\,dm_{K_q}(u).
\]
By the strong continuity of left translations on $L^1(K_q)$, the
function $J_q$ is continuous.  Indeed,
\[
|J_q(a_q)-J_q(a_q')|
\le
\|L_{a_q}\psi_q-L_{a_q'}\psi_q\|_{L^1(K_q)},
\]
where
$
L_a\psi(u):=\psi(au).
$
Moreover, using $m_q=(\tau_q)_*m_{K_q}$ and
\eqref{eq:local-realization-source-mass}, we obtain
\begin{equation*}
J_q(e)
=
\int_{\widetilde W_q}
\psi_q(u)\,dm_{K_q}(u)
=
\int_{W_q}
\mu_{q,z}(U_q)\,dm_q(z)
=
\mu_q\bigl(U_q\cap\kappa_q^{-1}(W_q)\bigr)
>0,
\end{equation*}
which together with the continuity of each $J_q$ and Tonelli's theorem implies that
\begin{align*}
\int_{\widetilde W_1\times\cdots\times\widetilde W_N}
\Lambda_t(U_1\times\cdots\times U_N)\,dm_P(t)
&=
\int_L
\prod_{q=1}^N
\left(
\int_{\widetilde W_q}
\psi_q(\ell_qt_q)\,dm_{K_q}(t_q)
\right)
dm_L(\ell)
\\
&=
\int_L
\prod_{q=1}^N J_q(\ell_q)\,dm_L(\ell)
>0.
\end{align*}

Therefore the Borel set
\[
B_{\mathbf U}
:=
\left\{
t\in
\widetilde W_1\times\cdots\times\widetilde W_N:
\Lambda_t(U_1\times\cdots\times U_N)>0
\right\}
\]
has positive $m_P$-measure.
By Proposition~\ref{prop:mixed-phase-decomposition}, the set
\[
E
:=
\{t\in P:\Lambda_t\text{ is ergodic}\}
\]
 is a Borel
conull subset of $P$.  Thus,
we may choose
$
t\in B_{\mathbf U}\cap E,
$
and put
$
\Lambda:=\Lambda_t.
$
Then $\Lambda$ is ergodic and
\begin{equation}\label{eq:local-realization-source-positive}
\Lambda(U_1\times\cdots\times U_N)>0.
\end{equation}
Since $t_q\in\widetilde W_q$, we have
$
t_qH_q\in W_q$, 
$q=1,\ldots,N$,
and hence  by \eqref{eq:local-realization-target-stability}, 
\begin{equation*}
\Lambda(C_1\times\cdots\times C_N)>\theta.
\end{equation*}

Apply Lemma~\ref{lem:mean-ergodic} to the ergodic system
$(X^N,\Lambda,G)$ and the function
$
\mathbf 1_{C_1\times\cdots\times C_N}.
$
The corresponding averages converge in $L^2(\Lambda)$ to the constant
$
\Lambda(C_1\times\cdots\times C_N).
$
After passing to a subsequence $(F_{n_j})_{j\ge1}$, the convergence
holds pointwise on a $\Lambda$-conull set.  By \eqref{eq:local-realization-source-positive},
we may choose
$
\mathbf y=(y_1,\ldots,y_N)
\in U_1\times\cdots\times U_N
$
in this convergence set.  Then
\[
\lim_{j\to\infty}
\frac{
\left|
\left\{
g\in F_{n_j}:
gy_q\in C_q
\text{ for every }q=1,\ldots,N
\right\}
\right|
}{|F_{n_j}|}
=
\Lambda(C_1\times\cdots\times C_N)
>
\theta.
\]
Taking the limsup along the original F{\o}lner sequence proves the claim.
\end{proof}

We have stated the preceding theorem in this general form because it may
be of independent interest.  For the remainder of the paper, however, we
only need the following consequence.
\begin{corollary}
\label{cor:phase-localization}
Fix $x\in X$ and  
$
z_q^0=t_q^0H_q\in A_q(x),$ $
q\in[N],
$
and put
$
t^0=(t_1^0,\ldots,t_N^0).
$
Assume that
\[
\Lambda_{t^0}
\bigl(X^N\setminus\Delta_{\mathrm{fat}}^N(X)\bigr)>0.
\]
Then there exists $\varepsilon_0>0$ such that,  for
every  F{\o}lner sequence $\Fcal$, $x$ is an $\Fcal$-mean $N$-sensitive point.
\end{corollary}

\begin{proof}
Since $X^N\setminus\Delta_{\mathrm{fat}}^N(X)$ is open and has
positive $\Lambda_{t^0}$-measure, choose
\[
\mathbf w=(w_1,\ldots,w_N)
\in
\supp(\Lambda_{t^0})
\setminus
\Delta_{\mathrm{fat}}^N(X).
\]
Choose open neighborhoods $O_q\ni w_q$ with pairwise disjoint closures,
and put
\[
a
:=
\Lambda_{t^0}(O_1\times\cdots\times O_N),
\qquad
c
:=
\min_{1\le p<q\le N}
\dist(\overline O_p,\overline O_q).
\]
Then $a>0$ and $c>0$.  Set
\[
\varepsilon_0:=\frac{ac}{2}>0.
\]

Fix a F{\o}lner sequence $\Fcal=(F_n)_{n\ge1}$ and a neighborhood
$U$ of $x$.  Apply
Theorem~\ref{thm:mixed-phase-local-realization} with all source points
equal to $x$, all source neighborhoods equal to $U$, target sets
$C_q=O_q$, and $\theta=a/2$.  We obtain
$y_1,\ldots,y_N\in U$ such that
\[
\limsup_{n\to\infty}
\frac{
\left|
\left\{
g\in F_n:
gy_q\in O_q
\text{ for every }q=1,\ldots,N
\right\}
\right|
}{|F_n|}
>
\frac a2.
\]
Whenever $gy_q\in O_q$ for every $q$, the definition of $c$ gives
\[
D_N(gy_1,\ldots,gy_N)\ge c.
\]
Therefore
\begin{align*}
D_{N,\Fcal}(y_1,\ldots,y_N)
&=
\limsup_{n\to\infty}
\frac1{|F_n|}
\sum_{g\in F_n}
D_N(gy_1,\ldots,gy_N)
\\
&\ge
c\,
\limsup_{n\to\infty}
\frac{
\left|
\left\{
g\in F_n:
gy_q\in O_q
\text{ for every }q
\right\}
\right|
}{|F_n|}
\\
&>
\frac{ac}{2}
=
\varepsilon_0.
\end{align*}
 Hence $x$ is $\Fcal$-mean $N$-sensitive  for every F{\o}lner sequence
$\Fcal$.
\end{proof}
\subsection{Multiplicity obstructions}
\label{subsec:multiplicity-obstructions}
We now use Corollary~\ref{cor:phase-localization} to rule out two sources
of $r$-fold separation.  The first consists of $r$ distinct
measure--phase labels compatible with the same point of $X$.  The second
consists of $r$ residual support points over one fixed compact phase.
In either case, the associated mixed phase joining gives positive mass to
the complement of the fat diagonal, and hence produces an
$\Fcal$-mean $r$-sensitive point.

For $x\in X$, put
\[
\mathcal A(x)
:=
\bigsqcup_{\mu\in\M_G^e(X)}
\{\mu\}\times A_\mu(x).
\]
\begin{lemma}
\label{lem:label-budget}
Assume that $(X,G)$ is
$\Fcal$-mean $r$-equicontinuous for some $r\ge2$.
Then for every $x\in X$,
\[
 |\mathcal A(x)|\le r-1.
\]
In particular, $(X,G)$ has at most $r-1$ ergodic invariant measures.
\end{lemma}
\begin{proof}
Suppose, for a contradiction, that $x\in X$ carries $r$ distinct
elements
\[
(\mu_1,z_1),\ldots,(\mu_r,z_r)
\in\mathcal A(x).
\]
Choose lifts
$
z_q=t_qH_q,$ $q=1,\ldots,r,
$
and put
$
t:=(t_1,\ldots,t_r).
$

Distinct ergodic invariant measures are mutually singular.  Hence, for
each distinct measure $\mu$ occurring in the list, we may choose a
$G$-invariant Borel set $E_\mu\subseteq X$ such that
$
\mu(E_\mu)=1
$
and such that the sets $E_\mu$ are pairwise disjoint.  Let $Z_\mu^\sharp\subseteq Z_\mu$ be a conull Borel set on which
\[
\mu_{\mu,z}(E_\mu)=1
\qquad\text{and}\qquad
\mu_{\mu,z}\bigl(\kappa_\mu^{-1}(\{z\})\bigr)=1.
\]

Since the $q$-th marginal of $\lambda_t$ is $m_{\mu_q}$, we have
\[
\lambda_t
\bigl(
Z_{\mu_1}^\sharp\times\cdots\times Z_{\mu_r}^\sharp
\bigr)
=
1.
\]
By the definition of $\lambda_t$, it follows that, for
$m_L$-a.e.
$\ell=(\ell_1,\ldots,\ell_r)\in L$,
\[
\ell_qt_qH_{\mu_q}\in Z_{\mu_q}^\sharp
\qquad
q\in[r].
\]
Fix such an $\ell$, and consider the product conditional measure
$
\Xi_\ell
:=
\bigotimes_{q=1}^r
\mu_{q,\ell_qt_qH_q}.
$
We claim that
\[
\Xi_\ell
\bigl(
X^r\setminus\Delta_{\mathrm{fat}}^r(X)
\bigr)
=
1.
\]
Indeed, if $\mu_p\ne\mu_q$, then the $p$-th and $q$-th coordinates are
concentrated on the disjoint sets $E_{\mu_p}$ and $E_{\mu_q}$,
respectively, and hence cannot coincide.

Suppose that $\mu_p=\mu_q=\mu$.  Since the labels
$(\mu_p,z_p)$ and $(\mu_q,z_q)$ are distinct, one has
$z_p\ne z_q$.  The $p$-th and $q$-th coordinate homomorphisms in
the joint phase construction are both equal to
$\alpha_\mu:G\to K_\mu$. Hence the group $L$ defined in \eqref{eq:phase-groups} satisfies
\[
L
=
\overline{
\left\{
\bigl(
\alpha_{\mu_1}(g),\ldots,\alpha_{\mu_r}(g)
\bigr):
g\in G
\right\}}
\subseteq
\{\ell:\ell_p=\ell_q\}.
\]
Therefore,
\[
\ell_p z_p\ne\ell_q z_q.
\]
The corresponding conditional measures are concentrated on two distinct
fibers of $\kappa_\mu$, so the $p$-th and $q$-th coordinates cannot
coincide.

Thus
\[
\Xi_\ell\bigl(X^r\setminus \Delta_{\mathrm{fat}}^r(X)\bigr)=1
\qquad\text{for $m_L$-a.e. $\ell$.}
\]
  Integrating over $L$, we obtain
\[
\Lambda_t\bigl(X^r\setminus \Delta_{\mathrm{fat}}^r(X)\bigr)=1.
\]
Since $z_q\in A_{\mu_q}(x)$, $q=1,\ldots,r$,
Corollary~\ref{cor:phase-localization} implies that $x$ is an
$\Fcal$-mean $r$-sensitive point.  This contradicts
Lemma~\ref{lem:no-sensitive-point}.
\end{proof}

\begin{lemma}
\label{lem:residual-multiplicity}
Assume that $(X,G)$ is
$\Fcal$-mean $r$-equicontinuous for some $r\ge2$. Then for every  $\mu\in\mc M_G^e(X)$, there is an integer
$
 b_\mu\in[r-1]
$
such that
\[
|\supp(\mu_{\mu,z})|=b_\mu
 \qquad
 \text{for $m_\mu$-a.e. }z\in Z_\mu.
\]
\end{lemma}

\begin{proof}
Since the support-cardinality function is measurable and it is
invariant under the ergodic action on $Z_\mu$, it follows that it is almost everywhere
constant, with value in $\N\cup\{\infty\}$.

Suppose that this value is at least $r$.  Fix $x\in X$ and 
$z^0=t^0H_\mu\in A_\mu(x)$.  Use $r$ copies of the same measure and the
phase $(t^0,\ldots,t^0)$.  In this repeated model the phase subgroup is
diagonal, so all compact-factor base points coincide, and
\[
 \Lambda_{(t^0,\ldots,t^0)}(X^r\setminus\Delta_{\mr{fat}}^r(X))
 =
 \int_{Z_\mu}
 \mu_{\mu,z}^{\otimes r}(X^r\setminus\Delta_{\mr{fat}}^r(X))
 \,dm_\mu(z)>0.
\]
Indeed, a probability measure whose support has at least $r$ points admits
$r$ pairwise disjoint open sets of positive measure, and its $r$-fold product
gives positive mass to their product.  Corollary~\ref{cor:phase-localization} now implies that $x$ is
$\Fcal$-mean $r$-sensitive, contradicting
Lemma~\ref{lem:no-sensitive-point}.  Hence
$
b_\mu\le r-1.
$
Since every probability measure has nonempty support,
$b_\mu\ge1$.
\end{proof}

\section{The exact degree formula and proof of the main theorem}
\label{sec:main-proof}

This section completes the proofs of Theorem \ref{thm:main} and Theorem \ref{thm:intro-degree}.  We first prove
that an upper bound on the conditional topomorphic degree implies the
corresponding Weyl mean equicontinuity.  For the converse, assuming
$\Fcal$-mean $r$-equicontinuity, we combine the compact-fiber
identification from the preceding section with the residual multiplicity
bound obtained from mixed-phase localization.  This yields the exact
formula
$
\tdeg_G(X)
=
\sum_{i=1}^{\ell}\iota_i b_i.
$
We then apply the localization theorem to the full multiplicity packet to
show that
$
\sum_{i=1}^{\ell}\iota_i b_i\le r-1
$, which completes the proof of Theorem \ref{thm:main}.

\subsection{Finite degree implies Weyl mean equicontinuity}

We begin with the direction that extends
\cite[Theorem~1.1]{BreitenbuecherHauptJaeger2026} to countably infinite
discrete amenable group actions.

\begin{theorem}
\label{thm:forward}
Let $(X,G)$ be a minimal $G$-system, and let $m\in\mb N$.  If
$
\tdeg_G(X)\le m,
$
then $(X,G)$ is Weyl mean $(m+1)$-equicontinuous.  
\end{theorem}
\begin{proof}
Put $d:=\tdeg_G(X)\le m$.  By
Proposition~\ref{prop:finite-conditional-degree}, there exist Borel maps
$
\gamma_1,\ldots,\gamma_d:\Xeq\to X
$
and a Borel set $Y_0\subseteq \Xeq$ with
$
\nu(Y_0)=1
$
such that
\[
\pi(\gamma_j(y))=y,
\qquad
y\in \Xeq,\ 1\le j\le d.
\]
Moreover, for every  $\mu\in\M_G(X)$, disintegrating it over $\pi$ by
$
\mu=\int_{\Xeq}\mu_y\,d\nu(y),
$
one has
\begin{equation}\label{eq:support}
\supp(\mu_y)
\subseteq
\{\gamma_1(y),\ldots,\gamma_d(y)\}\qquad\text{for $\nu$-a.e. $y\in Y_0$.}
\end{equation}

If $X$ is a singleton, the conclusion is immediate.  Otherwise, fix
$\varepsilon>0$ and rescale the metric so that $\diam(X)=1$.
Choose $\eta>0$ such that
\begin{equation}\label{eq:eta-choice}
(m+1)\eta <\frac{\varepsilon}{2}.
\end{equation}
By Lusin's theorem and inner regularity, there exists a compact set
$
K\subseteq Y_0
$
such that
$
\nu(K)>1-\eta
$
and every restriction
$
\gamma_j|_K
$
is continuous for $j\in[d]$.  Since there are only finitely many branches, there exists $\alpha>0$ such that
\begin{equation}\label{eq:branch-uniform-continuity}
y,y'\in K \text{ with }
d_{\Xeq}(y,y')<\alpha
\implies
d_X(\gamma_j(y),\gamma_j(y'))<\frac{\varepsilon}{6}\qquad\text{for every $j\in[d]$.}
\end{equation}

Since $\pi:X\to\Xeq$ is uniformly continuous, there exists
$\sigma>0$ such that
\begin{equation}\label{eq:pi-uniform-continuity}
d_X(u,v)<\sigma
\quad\Longrightarrow\quad
d_{\Xeq}(\pi(u),\pi(v))<\frac{\alpha}{3}.
\end{equation}
Choose
$
0<\rho<\min\left\{\sigma,\frac{\varepsilon}{6}\right\}.$
Put
\[
C
:=
\bigcup_{j=1}^d\gamma_j(K),
\qquad
U:=B_\rho(C).
\]
The set $C$ is compact.  By Urysohn's lemma, choose $f\in C(X)$ such that
\[
0\le f\le\mathbf 1_U,
\qquad
f|_C=1.
\]

Let $\mu\in\M_G(X)$. By \eqref{eq:support} 
for $\nu$-a.e. $y\in K$, the measure $\mu_y$ is supported on
\[
\{\gamma_1(y),\ldots,\gamma_d(y)\}\subseteq C.
\]
Consequently,
\[
\mu(C)
=
\int_{\Xeq}\mu_y(C)\,d\nu(y)
\ge\nu(K)>1-\eta.
\]
It follows that
\[
\inf_{\mu\in\mc M_G(X)}\int_X f\,d\mu>1-\eta.
\]
By Lemma~\ref{lem:uniform-lower-bound}, for every $x\in X$ and every
 F{\o}lner sequence,
\begin{equation}\label{eq:uniform-hit-density}
\liminf_{n\to\infty}
\frac{
|\{g\in F_n:gx\in U\}|
}{
|F_n|
}
\ge\liminf_{n\to\infty}
\frac1{|F_n|}
\sum_{g\in F_n}f(gx)
\ge1-\eta.
\end{equation}

Now let
$
\mathbf x=(x_1,\ldots,x_{m+1})\in X^{m+1}
$
satisfy
\begin{equation}\label{eq:initial-x-closeness}
\max_{1\le i<j\le m+1}d_X(x_i,x_j)<\sigma.
\end{equation}
Fix an arbitrary F{\o}lner sequence $\Fcal$. For each $n\ge1$, put
\[
G_n
:=
\{g\in F_n:gx_i\in U\text{ for every }i\in[m+1]\}.
\] By
\eqref{eq:uniform-hit-density} and \eqref{eq:eta-choice}, 
\begin{equation}\label{eq:common-hit-density}
    \liminf_{n\to\infty}
\frac{
|G_n|
}{
|F_n|
}\ge 
1-(m+1)\eta> 1-\varepsilon/2.
\end{equation}
Fix $n\ge1$ and $g\in G_n$. For each $i$, choose
$
j_i\in[d],$
$
y_i\in K,
$
such that
$
d_X(gx_i,\gamma_{j_i}(y_i))<\rho.
$
For $i,i'\in[m+1]$, by \eqref{eq:pi-uniform-continuity}, we have
\begin{align*}
d_{\Xeq}(y_i,y_{i'})
&\le
d_{\Xeq}(y_i,\pi(gx_i))
+
d_{\Xeq}(\pi(gx_i),\pi(gx_{i'}))
+
d_{\Xeq}(\pi(gx_{i'}),y_{i'})
\\
&<
\frac{\alpha}{3}
+
d_{\Xeq}(\pi(x_i),\pi(x_{i'}))
+
\frac{\alpha}{3}
<
\alpha.
\end{align*}
Here we used
$
\pi(\gamma_{j_i}(y_i))=y_i
$
and the $G$-invariance of the metric on $\Xeq$.

Since there are $m+1$ indices $j_i$ but only $d\le m$ branches, there
exist $a\ne b$ such that
$
j_a=j_b=:j.
$
By \eqref{eq:branch-uniform-continuity},
\[
d_X(\gamma_j(y_a),\gamma_j(y_b))
<
\varepsilon/6.
\]
Therefore
\begin{align*}
d_X(gx_a,gx_b)
&\le
d_X(gx_a,\gamma_j(y_a))
+
d_X(\gamma_j(y_a),\gamma_j(y_b))+
d_X(\gamma_j(y_b),gx_b)\\
&<
2\rho+{\varepsilon}/{6}
<
{\varepsilon}/{2}.
\end{align*}
Thus
$
D_{m+1}(g\mathbf x)<{\varepsilon}/{2}
$
for all $g\in G_n$.  At all remaining times,
$
D_{m+1}(g\mathbf x)\le \diam(X)=1.
$
It follows from \eqref{eq:common-hit-density} that
\begin{align*}
D_{m+1,\Fcal}(\mathbf x)
<
{\varepsilon}/{2}
+
{\varepsilon}/{2}=
\varepsilon.
\end{align*}

The constants $\sigma$ are independent of the choice of
$\Fcal$. 
Hence $(X,G)$ is Weyl mean $(m+1)$-equicontinuous.
\end{proof}

\subsection{Proof of Theorem \ref{thm:intro-degree}}
Theorem \ref{thm:intro-degree} will be obtained by combining Proposition~\ref{prop:ergodic-degree-contribution} and
Theorem~\ref{thm:general-degree-decomposition} in the following.

For $\mu\in\M_G^e(X)$, retain the factorization
\[
 (X,\mu,G)\xrightarrow{\ \kappa_\mu\ }
 (Z_\mu,m_\mu,G)\xrightarrow{\ p_\mu\ }(\Xeq,\nu,G),
\]
and write
\[
 \mu=\int_{Z_\mu}\mu_{\mu,z}\,dm_\mu(z),
 \qquad
 m_\mu=\int_{\Xeq}(m_\mu)_y\,d\nu(y).
\]
Set
\[
\iota_\mu
:=
\bigl|p_\mu^{-1}(y)\bigr|,
\qquad
b_\mu
:=
\bigl|\supp(\mu_{\mu,z})\bigr|
\quad\text{for $m_\mu$-a.e. }z\in Z_\mu.
\]
Here $\iota_\mu$ is independent of
$y\in X_{\mathrm{eq}}$ by
Lemma~\ref{lem:homogeneous-factor-fibers}, while $b_\mu$ is
$m_\mu$-a.e. constant by
Proposition~\ref{prop:residual-sequence-entropy}.

\begin{proposition}\label{prop:ergodic-degree-contribution}
For every $\mu\in\M_G^e(X)$,
\[
 |\supp(\mu_y)|=\iota_\mu b_\mu
 \quad\text{for $\nu$-a.e. }y\in\Xeq.
\]
  Consequently,
\[
 |\supp(\mu_y)|
 =\iota_\mu\exp\bigl(h_\mu^*(G)\bigr)
 \quad\text{for $\nu$-a.e. }y\in\Xeq.
\]
\end{proposition}

\begin{proof}
The composition rule for disintegrations gives
\begin{equation}\label{eq:general-composed-disintegration}
 \mu_y
 =\int_{p_\mu^{-1}(y)}\mu_{\mu,z}\,d(m_\mu)_y(z),
 \qquad
 (\kappa_\mu)_*\mu_y=(m_\mu)_y \quad\text{for $\nu$-a.e. }y\in\Xeq.
\end{equation}

Assume first that $\iota_\mu<\infty$.  By
Lemma~\ref{lem:homogeneous-factor-fibers}, $(m_\mu)_y$ is the uniform
probability measure on $p_\mu^{-1}(y)$.  By Proposition \ref{prop:residual-sequence-entropy}, we may let $Z_\mu^0$ be an
$m_\mu$-conull Borel set on which
\[
|\supp(\mu_{\mu,z})|=b_\mu,
 \qquad
 \mu_{\mu,z}\bigl(\kappa_\mu^{-1}(\{z\})\bigr)=1.
\]
Disintegrating $m_\mu(Z_\mu^0)=1$ over $p_\mu$ and using uniformity on
the finite fibers, we obtain
$p_\mu^{-1}(y)\subseteq Z_\mu^0$ for $\nu$-a.e. $y$.
If $b_\mu<\infty$, every point in $\supp(\mu_{\mu,z})$ is an atom of
positive mass and therefore belongs to $\kappa_\mu^{-1}(\{z\})$.
Thus the supports corresponding to distinct $z\in p_\mu^{-1}(y)$ are
disjoint, and \eqref{eq:general-composed-disintegration} yields
\[
 \supp(\mu_y)
 =\bigsqcup_{z\in p_\mu^{-1}(y)}\supp(\mu_{\mu,z}),
 \qquad
|\supp(\mu_y)|=\iota_\mu b_\mu.
\]

If instead $b_\mu=\infty$, every component in the finite uniform sum
already has infinite support, and hence so does $\mu_y$.
Now assume that $\iota_\mu=\infty$.  The conditional measure
$(m_\mu)_y$ has full support on the infinite compact fiber
$p_\mu^{-1}(y)$.  If $\mu_y$ had finite support, then its push-forward
$(\kappa_\mu)_*\mu_y$ would also have finite support, contradicting the
second identity in \eqref{eq:general-composed-disintegration}.  Hence
$|\supp(\mu_y)|=\infty=\iota_\mu b_\mu$.

Finally, Proposition~\ref{prop:residual-sequence-entropy} gives
$b_\mu=\exp(h_\mu^*(G))$.
\end{proof}

\begin{theorem}\label{thm:general-degree-decomposition}
One has
\[
 \tdeg_G(X)
 =\sum_{\mu\in\M_G^e(X)}\iota_\mu b_\mu.
\]
\end{theorem}

\begin{proof}
Put $k_\mu:=\iota_\mu b_\mu$.  If $k_\mu=\infty$ for some ergodic
$\mu$, Proposition~\ref{prop:ergodic-degree-contribution} immediately
gives $\tdeg_G(X)=\infty$.  We may therefore suppose, when proving the
nontrivial lower bound, that the measures under consideration have finite
$k_\mu$.

Let $F\subseteq\M_G^e(X)$ be finite.  Distinct ergodic invariant measures
are mutually singular, so there are pairwise disjoint invariant Borel sets
$(E_\mu)_{\mu\in F}$ with $\mu(E_\mu)=1$.  Since the conditional measures
of every $\mu\in F$ over $\pi$ have finite support, their support points
are atoms and hence lie in $E_\mu$ for $\nu$-a.e. $y$.  The conditional
measure of
$
 \theta_F:=\frac1{|F|}\sum_{\mu\in F}\mu
$
is $|F|^{-1}\sum_{\mu\in F}\mu_y$, and the component supports are
pairwise disjoint.  Therefore
\[
 |\supp((\theta_F)_y)|=\sum_{\mu\in F}k_\mu
 \quad\text{for $\nu$-a.e. }y\in\Xeq.
\]
Taking the supremum over finite $F$ gives
\[
 \tdeg_G(X)
 \ge\sum_{\mu\in\M_G^e(X)}k_\mu.
\]

If the extended sum on the right is infinite, this already proves
$\tdeg_G(X)=\infty$.  If it is finite, then there are only finitely many
ergodic measures, say $\mu_1,\ldots,\mu_\ell$, and every $k_{\mu_i}$ is
finite.  Each $\theta\in\M_G(X)$ has the form
$\theta=\sum_{i=1}^\ell c_i\mu_i$.  On a common conull set,
\[
 \supp(\theta_y)
 \subseteq\bigcup_{\{i:c_i>0\}}\supp((\mu_i)_y),
\]
so
\[
|\supp(\theta_y)|
 \le\sum_{i=1}^\ell k_{\mu_i}.
\]
Taking the supremum over $\theta$ gives the reverse inequality. 
\end{proof}

\subsection{The converse degree bound}
We now show that conditional topomorphic degree at least $r$ produces
a local mean $r$-sensitive tuple, contradicting mean
$r$-equicontinuity.

\begin{theorem}
\label{thm:converse-degree-bound}
If $(X,G)$ is $\Fcal$-mean $r$-equicontinuous for some 
F{\o}lner sequence $\Fcal$, then
\[
\tdeg_G(X)\le r-1.
\]
\end{theorem}

\begin{proof}
By Lemmas~\ref{lem:label-budget} and
\ref{lem:residual-multiplicity}, together with
Corollary~\ref{cor:automatic-phase-persistence}, there are only finitely
many ergodic measures and each $\iota_i$ and $b_i$ is finite.
Therefore Theorem~\ref{thm:general-degree-decomposition} yields
\[
M
:=
\tdeg_G(X)
=
\sum_{i=1}^{\ell}\iota_i b_i
<\infty.
\]
Assume, towards a contradiction, that $M\ge r$.

Fix $x_0\in X$, put $y_0:=\pi(x_0)$, and, using
Corollary~\ref{cor:automatic-phase-persistence}, write
\[
p_i^{-1}(y_0)
=
\{z_{i,1},\ldots,z_{i,\iota_i}\}
=
A_i(x_0).
\]
For each $i$ and $j$, choose $t_{i,j}\in K_i$ such that
$
z_{i,j}=t_{i,j}H_i.
$
Let $t^0$ be the $M$-tuple obtained by repeating $t_{i,j}$ exactly
$b_i$ times for every pair $(i,j)$, and let $L^{(M)}$ be the
associated phase subgroup.  Since all coordinate copies carrying
$\mu_i$ use the same homomorphism
$\alpha_i:G\to K_i$, every $u\in L^{(M)}$ has a common
$K_i$-coordinate across these copies; denote it by $u_i$. Set
\[
z_{i,j}(u):=u_it_{i,j}H_i,
\qquad
u\in L^{(M)}.
\]
The corresponding mixed phase joining on $X^M$ is
\begin{equation}\label{eq:full-packet-phase-joining}
\Lambda_{t^0}
=
\int_{L^{(M)}}
\bigotimes_{i=1}^{\ell}
\bigotimes_{j=1}^{\iota_i}
\mu_{i,z_{i,j}(u)}^{\otimes b_i}
\,dm_{L^{(M)}}(u).
\end{equation}
Choose pairwise disjoint invariant Borel sets $E_i$ with $\mu_i(E_i)=1$, $i\in[\ell]$.
For each $i$, choose an $m_i$-conull Borel set
$\mathcal G_i\subseteq Z_i$ such that, for every
$z\in\mathcal G_i$,
\[
|\supp(\mu_{i,z})|=b_i
\qquad\text{and}\qquad
\mu_{i,z}
\bigl(E_i\cap\kappa_i^{-1}(\{z\})\bigr)=1.
\]

For fixed $i$, the points $z_{i,1}(u),\ldots,z_{i,\iota_i}(u)$
are pairwise distinct, since left translation by $u_i$ is a
bijection of $Z_i=K_i/H_i$.
The coordinate projection $L^{(M)}\to K_i$ is surjective, so it
pushes Haar measure on $L^{(M)}$ forward to normalized Haar measure
on $K_i$.  Consequently, for every $i,j$, the map
$
u\mapsto z_{i,j}(u)
$
has distribution $m_i$.  Since only finitely many pairs $(i,j)$ occur, for
$m_{L^{(M)}}$-a.e. $u\in L^{(M)}$,
\begin{equation}\label{eq:full-packet-good-phases}
z_{i,j}(u)\in\mathcal G_i
\qquad
\text{for all }i,j.
\end{equation}

Fix $u$ satisfying
\eqref{eq:full-packet-good-phases}, and denote the integrand in
\eqref{eq:full-packet-phase-joining} by
$
\Xi_u
:=
\bigotimes_{i=1}^{\ell}
\bigotimes_{j=1}^{\iota_i}
\mu_{i,z_{i,j}(u)}^{\otimes b_i}.
$
For each $i,j$, let
\[
q_{i,j}(u)
:=
\mu_{i,z_{i,j}(u)}^{\otimes b_i}
\bigl(X^{b_i}\setminus\Delta_{\mathrm{fat}}^{b_i}(X)\bigr).
\]
Since $\mu_{i,z_{i,j}(u)}$ has exactly $b_i$ support points, each
of positive mass, its $b_i$-fold product assigns positive mass to the
set of pairwise distinct $b_i$-tuples.  Hence
$
q_{i,j}(u)>0.
$

Coordinates belonging to different blocks cannot collide.  If two blocks
have the same index $i$ but different indices $j\ne j'$, their
conditional measures are concentrated on the disjoint fibers
\[
\kappa_i^{-1}(\{z_{i,j}(u)\})
\quad\text{and}\quad
\kappa_i^{-1}(\{z_{i,j'}(u)\}).
\]
If the blocks have different indices $i\ne i'$, their conditional
measures are concentrated on the disjoint sets $E_i$ and $E_{i'}$.
It follows that
\[
\Xi_u\bigl(X^M\setminus\Delta_{\mathrm{fat}}^M(X)\bigr)
=
\prod_{i=1}^{\ell}
\prod_{j=1}^{\iota_i}
q_{i,j}(u)>0.
\]
Integrating the conditional measures defining the phase joining,
we obtain
\[
\Lambda_{t^0}\bigl(X^M\setminus\Delta_{\mathrm{fat}}^M(X)\bigr)
=
\int_{L^{(M)}}
\Xi_u\bigl(X^M\setminus\Delta_{\mathrm{fat}}^M(X)\bigr)\,
dm_{L^{(M)}}(u)
>0.
\]

Every prescribed phase $z_{i,j}$ belongs to $A_i(x_0)$.
Therefore Corollary~\ref{cor:phase-localization} yields
$\varepsilon>0$ such that every neighborhood of $x_0$ contains
an $M$-tuple
$
\mathbf x=(x_1,\ldots,x_M)
$
with
$
D_{M,\Fcal}(\mathbf x)\ge\varepsilon.
$
Thus $x_0$ is an $\Fcal$-mean $M$-sensitive point and hence, since
$M\ge r$, also an $\Fcal$-mean $r$-sensitive point. This
contradicts Lemma~\ref{lem:no-sensitive-point}.  Therefore $\tdeg_G(X)=M\le r-1 $.
\end{proof}

\subsection{Completion of the proof of Theorem~\ref{thm:main}}

\begin{proof}[Proof of Theorem~\ref{thm:main}]
The implication
$\textup{(iii)}\Rightarrow\textup{(i)}$ follows from
Theorem~\ref{thm:forward}, applied with $m=r-1$.
Since
$
D_{r,\Fcal}(\mathbf x)
\le
D_r^{\mathrm W}(\mathbf x)
$
for every F{\o}lner sequence $\Fcal$, we have
$\textup{(i)}\Rightarrow\textup{(ii)}$.
Finally,
$\textup{(ii)}\Rightarrow\textup{(iii)}$ is
Theorem~\ref{thm:converse-degree-bound}.
\end{proof}

\section{IT-tuples and sequence-entropy consequences}
\label{sec:sequence-entropy}

This section proves Theorem~\ref{thm:intro-it-tuples}.  For
$\mu\in\M_G^e(X)$, retain the notation
\[
(X,\mu,G)
\overset{\kappa_\mu}{\longrightarrow}
(Z_\mu,m_\mu,G)
\overset{p_\mu}{\longrightarrow}
(\Xeq,\nu,G)
\]
from the preceding sections.  Let
\[
b_\mu
:=
|\supp(\mu_{\mu,z})|
\quad
\text{for $m_\mu$-a.e. }z\in Z_\mu\qquad\text{and}\qquad
\iota_\mu
:=
|p_\mu^{-1}(y)|,
\quad
y\in\Xeq.
\]

\subsection{Mixed phase supports and IT-tuples}
\label{subsec:mixed-phases-IT}
The construction below is a multivariate phase version of the branching
argument developed in \cite{LiuXuZhang2025Measure}.
Fix ergodic measures
$\mu_1,\ldots,\mu_N$, compact homogeneous models
$Z_q=K_q/H_q$, and nonempty open sets $V_q\subseteq X$.
For a nonempty finite set $B$ and phases
\[
z_{\beta,q}=t_{\beta,q}H_q\in Z_q,
\qquad
(\beta,q)\in B\times[N],
\]
define
\begin{equation}\label{eq:phase-correlation}
\Phi_B\bigl((z_{\beta,q})_{\beta,q}\bigr)
:=
\int_L
\prod_{\beta\in B}
\prod_{q=1}^N
\mu_{q,\ell_qt_{\beta,q}H_q}(V_q)
\,dm_L(\ell),
\end{equation}
where
\[
L
:=
\overline{
\left\{
(\alpha_1(g),\ldots,\alpha_N(g)):
g\in G
\right\}
}.
\]
The value in \eqref{eq:phase-correlation} is independent of the chosen
lifts.  It is the mass of the target rectangle associated with the
repeated coordinate list $B\times[N]$, and is continuous in the phases
by Lemma~\ref{lem:rectangle-continuity}.

\begin{proposition}
\label{prop:mixed-phases-produce-IT}
Fix $x_0\in X$ and 
$
z_q^0=t_q^0H_q\in A_q(x_0),$ $
q\in[N],
$ and put $t^0=(t_1^0,\ldots,t_N^0)$. 
Let $\Lambda_{t^0}$ be the associated mixed phase joining.  Then
\[
\supp(\Lambda_{t^0})
\subseteq
\IT_N(X,G).
\]
\end{proposition}

\begin{proof}
Fix
$\mathbf w=(w_1,\ldots,w_N)\in\supp(\Lambda_{t^0})$
and neighborhoods $V_q\ni w_q$.  Then
\begin{equation}\label{eq:initial-correlation}
\Phi_{\{\emptyset\}}
\bigl((z_q^0)_{q=1}^N\bigr)
=
\Lambda_{t^0}(V_1\times\cdots\times V_N)
>0.
\end{equation}
Fix a F{\o}lner sequence
$\Fcal=(F_n)_{n\ge1}$.  We use it to construct an infinite
independence set for $(V_1,\ldots,V_N)$.

We use the following inductive step.  Let $B$ be finite and nonempty.
Suppose that, for every $\beta\in B$, we have chosen
\[
x_\beta\in W_\beta,
\qquad
W_\beta\subseteq X\text{ nonempty and open},
\qquad
z_{\beta,q}\in A_q(x_\beta)
\quad(q\in[N]),
\]
and that
\begin{equation}\label{eq:inductive-correlation}
\Phi_B\bigl((z_{\beta,q})_{\beta,q}\bigr)>0.
\end{equation}

Let $E_0\subseteq G$ be finite.  We shall find
$g_*\in G\setminus E_0$ and, for every
$(\beta,k)\in B\times[N]$, 
\[
x_{\beta k}\in W_{\beta k}
\subseteq
W_\beta\cap g_*^{-1}V_k,
\qquad
z_{\beta k,q}\in A_q(x_{\beta k}),
\]
such that
\begin{equation}\label{eq:child-phase-correlation}
\Phi_{B\times[N]}
\bigl((z_{\beta k,q})_{\beta,k,q}\bigr)>0.
\end{equation}

Set
\[
\widetilde z_{\beta k,q}:=z_{\beta,q},
\qquad
(\beta,k,q)\in B\times[N]\times[N].
\]
Put
\[
F(\ell)
:=
\prod_{\beta\in B}
\prod_{q=1}^N
\mu_{q,\ell_qt_{\beta,q}H_q}(V_q).
\]
Then $0\le F\le1$, and
\eqref{eq:inductive-correlation} gives
$\int_LF\,dm_L>0$.  Hence
\[
\Phi_{B\times[N]}
\bigl((\widetilde z_{\beta k,q})_{\beta,k,q}\bigr)
=
\int_LF(\ell)^N\,dm_L(\ell)
>0.
\]
By continuity, there are neighborhoods
$\mathcal O_{\beta,q}\ni z_{\beta,q}$ such that
\begin{equation}\label{eq:stable-cloned-correlation}
z'_{\beta k,q}\in\mathcal O_{\beta,q}
\text{ for all }\beta,k,q
\quad\Longrightarrow\quad
\Phi_{B\times[N]}
\bigl((z'_{\beta k,q})_{\beta,k,q}\bigr)>0.
\end{equation}
By Lemma~\ref{lem:automatic-fiber-continuity}, every $A_q$ is
Hausdorff-continuous.  Since
$z_{\beta,q}\in A_q(x_\beta)$, for each $\beta\in B$ there is a
nonempty open neighborhood $S_\beta\subseteq W_\beta$ of $x_\beta$
such that
\begin{equation}\label{eq:phase-persistence}
A_q(y)\cap\mathcal O_{\beta,q}\ne\emptyset
\qquad
(y\in S_\beta,\ q\in[N]).
\end{equation}

Apply Theorem~\ref{thm:mixed-phase-local-realization} to the repeated
coordinate list $B\times[N]$.  Coordinate $(\beta,k)$ carries the
measure $\mu_k$, the phase
$
z_{\beta,k}\in A_k(x_\beta),
$
the source neighborhood $S_\beta$, and the target set $V_k$.
For this repeated list, the target rectangle has mass exactly
$
\Phi_B\bigl((z_{\beta,q})_{\beta,q}\bigr)>0.
$
Therefore, Theorem~\ref{thm:mixed-phase-local-realization} gives points $y_{\beta,k}\in S_\beta$ such that
\[
R
:=
\left\{
g\in G:
gy_{\beta,k}\in V_k
\text{ for every }(\beta,k)\in B\times[N]
\right\}
\]
has positive upper $\Fcal$-density.  In particular, $R$ is infinite,
so we may choose
$
g_*\in R\setminus E_0.
$

For every $(\beta,k)\in B\times[N]$, choose a nonempty open
neighborhood $W_{\beta k}$ of $y_{\beta,k}$ such that
$
W_{\beta k}
\subseteq
S_\beta\cap g_*^{-1}V_k,
$
and put $x_{\beta k}:=y_{\beta,k}$.  By
\eqref{eq:phase-persistence}, choose
\[
z_{\beta k,q}
\in
A_q(x_{\beta k})\cap\mathcal O_{\beta,q},
\qquad
q\in[N].
\]
Then the required nesting property holds, while
\eqref{eq:stable-cloned-correlation} gives
\eqref{eq:child-phase-correlation}.  This completes the inductive step.

We now iterate the inductive step.  Set
$
B_n:=[N]^n$, $n\ge0$,
with $B_0=\{\emptyset\}$, and start from
\[
x_\emptyset:=x_0,
\qquad
W_\emptyset:=X,
\qquad
z_{\emptyset,q}:=z_q^0.
\]
At stage $n$, apply the inductive step with
\[
B=B_n
\qquad\text{and}\qquad
E_0=\{g_1,\ldots,g_n\}.
\]
This produces pairwise distinct elements
$g_1,g_2,\ldots$ and nonempty open sets satisfying
\[
W_{a_1\cdots a_n}
\subseteq
\bigcap_{j=1}^n g_j^{-1}V_{a_j}\qquad\text{for every $n\ge1$ and every
$(a_1,\ldots,a_n)\in[N]^n$.}
\]
Then $J:=\{g_n:n\ge1\}.$ is an infinite independence set for
$(V_1,\ldots,V_N)$.  Since the neighborhoods $V_q$ were arbitrary,
$\mathbf w\in\IT_N(X,G)$.
\end{proof}

\begin{lemma}
\label{lem:finite-packet-off-diagonal}
For every  $N\in\mb N$ with $2\le N\le \tdeg_G(X)$ and every $x_0\in X$,
there exist ergodic measures
$
\mu_1,\ldots,\mu_N\in\M_G^e(X),
$
not necessarily distinct, and
$
z_q\in A_{\mu_q}(x_0),$ $q\in[N]$
such that the associated mixed compact-phase joining
$
\Lambda_{x_0}^{(N)}\in\Prob(X^N)
$
satisfies
\[
\Lambda_{x_0}^{(N)}
\bigl(
X^N\setminus\Delta_{\mathrm{fat}}^N(X)
\bigr)
>0.
\]
\end{lemma}

\begin{proof}
Since $N\le d$ and $N$ is finite, the definition of the extended
sum gives distinct ergodic measures
$
\mu_1,\ldots,\mu_s
$
and integers
$
1\le J_i\le\iota_{\mu_i},$  $
1\le n_{i,j}\le b_{\mu_i}$, for
$1\le j\le J_i$
such that
\[
\sum_{i=1}^s\sum_{j=1}^{J_i}n_{i,j}=N.
\]
Write
\[
Z_i=K_i/H_i,
\qquad
\kappa_i=\kappa_{\mu_i},
\qquad
m_i=m_{\mu_i},
\qquad
b_i=b_{\mu_i}.
\]

Put $y_0:=\pi(x_0)$.  By Corollary \ref{cor:automatic-phase-persistence}, choose distinct phases
\[
z_{i,1},\ldots,z_{i,J_i}
\in
p_{\mu_i}^{-1}(y_0)
=
A_{\mu_i}(x_0),
\]
and write $z_{i,j}=t_{i,j}H_i$.  For every pair $(i,j)$, take
$n_{i,j}$ copies of $\mu_i$, all carrying the phase $z_{i,j}$, and
let $\Lambda_{x_0}^{(N)}$ be the associated mixed phase joining.  Thus
\[
\Lambda_{x_0}^{(N)}
=
\int_{L^{(N)}}\Xi_u\,dm_{L^{(N)}}(u),
\qquad
\Xi_u
:=
\bigotimes_{i=1}^s
\bigotimes_{j=1}^{J_i}
\mu_{i,u_it_{i,j}H_i}^{\otimes n_{i,j}},
\]
where $u_i$ is the common $K_i$-coordinate of all copies carrying
$\mu_i$.

The remainder follows from the packet-separation argument used in the
proof of Theorem~\ref{thm:converse-degree-bound}.  Indeed, for
$m_{L^{(N)}}$-a.e. $u$, each conditional measure
$\mu_{i,u_it_{i,j}H_i}$ has support cardinality
$b_i\ge n_{i,j}$, and hence its $n_{i,j}$-fold product assigns
positive mass to
$
X^{n_{i,j}}
\setminus
\Delta_{\mathrm{fat}}^{n_{i,j}}(X).
$
Different blocks cannot collide: blocks with the same $i$ and
different $j$'s are carried by distinct $\kappa_i$-fibers, whereas
blocks with different $i$'s are carried by pairwise disjoint invariant
full-measure sets.  Consequently,
\[
\Xi_u
\bigl(
X^N\setminus\Delta_{\mathrm{fat}}^N(X)
\bigr)
>0
\]
for $m_{L^{(N)}}$-a.e. $u$.  Integrating over $L^{(N)}$
gives
\[
\Lambda_{x_0}^{(N)}
\bigl(
X^N\setminus\Delta_{\mathrm{fat}}^N(X)
\bigr)
>0.\qedhere
\]
\end{proof}

\subsection{Proof of Theorem \ref{thm:intro-it-tuples}} 
By \eqref{eq:maximal-pattern-IN-rank}, we immediately obtain the following lemma.
\begin{lemma}\label{lem:IT-entropy-lower-bound}
If $(X,G)$ admits an essential IT $N$-tuple, then
\[
h_{\mathrm{top}}^*(X,G)\ge\log N.
\]
\end{lemma}

\begin{proof}[Proof of Theorem~\ref{thm:intro-it-tuples}]
Put $d:=\tdeg_G(X)$.
Let $N\in\mb N$  with $2\le N\le d$.  By
Theorem~\ref{thm:intro-degree} and
Lemma~\ref{lem:finite-packet-off-diagonal}, there is a mixed phase joining
$\Lambda_{x_0}^{(N)}$ compatible with some $x_0\in X$ and satisfying
\[
\Lambda_{x_0}^{(N)}
\bigl(
X^N\setminus\Delta_{\mathrm{fat}}^N(X)
\bigr)
>0.
\]
Its support therefore contains an essential $N$-tuple.  By
Proposition~\ref{prop:mixed-phases-produce-IT}, this tuple is an IT
$N$-tuple.

If $d=1$, the entropy estimate is trivial.  If
$2\le d<\infty$, Lemma~\ref{lem:IT-entropy-lower-bound}, applied with
$N=d$, gives
\[
h_{\mathrm{top}}^*(X,G)\ge\log d.
\]
If $d=\infty$, then for every finite $N\ge2$ the preceding argument
gives an essential IT $N$-tuple, and hence
$
h_{\mathrm{top}}^*(X,G)\ge\log N.
$
Letting $N\to\infty$, we obtain
\[
h_{\mathrm{top}}^*(X,G)=\infty.\qedhere
\]
\end{proof}

\subsection{Exact uniquely ergodic realizations of finite sequence entropy}
\label{subsec:unique-ergodic-virtually-abelian}

We apply Theorem~\ref{thm:intro-it-tuples} to answer both parts of
\cite[Question~5.1]{GomezLeonTorresMunozLopez2026}.  The generalized
Morse system settles the $\mathbb Z$-action in \cite[Proposition~2.9]{GomezLeonTorresMunozLopez2026}, and
coordinate products followed by finite-index induction settle \cite[Theorem~1.2]{GomezLeonTorresMunozLopez2026}.

For $m\ge2$, let $(M_m,\sigma)$ be the generalized $m$-Morse subshift
associated with the constant-length substitution on
$\mathbb Z/m\mathbb Z$ given by
\[
 \tau_m(j):=j\,(j+1)\cdots(j+m-1),
\]
where the symbols are read modulo $m$.  This is a minimal uniquely
ergodic zero-entropy subshift; see, for example,
\cite[Chapter~5]{Queffelec2010}.  Maass and Shao
\cite[Section~6.2, Example~2]{MaassShao2007} exhibited a factor tower
\begin{equation}\label{eq:Morse-factor-tower}
 (M_m,\sigma)
 \xrightarrow{\ \pi_m\ }
 (Y_m,S_m)
 \xrightarrow{\ \rho_m\ }
 ((M_m)_{\mathrm{eq}},\sigma_{\mathrm{eq}}),
\end{equation}
where $\pi_m$ is exactly $m$-to-one, $(Y_m,S_m)$ is null,
$\rho_m$ is asymptotic, and $\rho_m\circ\pi_m$ is the maximal
equicontinuous factor map.

\begin{lemma}\label{lem:Morse-exact-degree-independence}
For every $m\ge2$,
\[
 \tdeg_{\mathbb Z}(M_m)=m,
 \qquad
 h_{\mathrm{top}}^*(M_m,\mathbb Z)=\log m.
\]
Moreover,
\[
 \IT_m(M_m,\mathbb Z)
 \setminus\Delta_{\mathrm{fat}}^m(M_m)\ne\emptyset,
 \qquad
 \IN_{m+1}(M_m,\mathbb Z)
 \subseteq\Delta_{\mathrm{fat}}^{m+1}(M_m).
\]
\end{lemma}

\begin{proof}
Let $\mu_m$ be the unique invariant measure of $M_m$ and put
$\eta_m:=(\pi_m)_*\mu_m$.  Since $\pi_m$ is exactly $m$-to-one, the
map from $Y_m\ni
y\mapsto u_y:=\frac1m\sum_{x\in\pi_m^{-1}(y)}\delta_x\in \Prob(M_m)$
is an equivariant Borel map.  Hence
$\int u_y\,d\eta_m(y)$ is $\sigma$-invariant and, by unique
ergodicity, equals $\mu_m$.  Thus the conditional measures of
$\mu_m$ over $\pi_m$ are uniform on the $m$-point fibers.

It remains to observe that the asymptotic factor
$\rho_m:Y_m\to (M_m)_{\mathrm{eq}}$ is measure-theoretically
one-to-one.  Indeed, if
$
\eta_m=\int(\eta_m)_z\,d\lambda_m(z)
$
is the disintegration over $\rho_m$, then the relative product
$
\xi:=\int(\eta_m)_z\otimes(\eta_m)_z\,d\lambda_m(z)
$
is $S_m\times S_m$-invariant and supported on the fiber relation of
$\rho_m$.  Since $\rho_m$ is asymptotic,
$
\lim_{n\to\infty}d(S_m^ny,S_m^ny')=0
$
for $\xi$-a.e.\ $(y,y')$.  Invariance and dominated convergence give
\[
\int d(y,y')\,d\xi(y,y')
=
\int d(S_m^ny,S_m^ny')\,d\xi(y,y')
\to0\qquad\text{ as }n\to\infty.
\]
Hence $\xi$ is concentrated on the diagonal, and therefore
$(\eta_m)_z$ is a Dirac measure for $\lambda_m$-a.e.\ $z$.

Composing the two disintegrations, the conditional measures of
$\mu_m$ over the maximal equicontinuous factor have exactly $m$
support points almost everywhere.  Consequently,
\[
\tdeg_{\mathbb Z}(M_m)=m.
\]
By Theorem~\ref{thm:intro-it-tuples}, $M_m$ admits an essential IT
$m$-tuple.  On the other hand, Maass and Shao \cite{MaassShao2007} proved that
$
h_{\mathrm{top}}^*(M_m,\mathbb Z)=\log m.
$
Hence \eqref{eq:maximal-pattern-IN-rank} rules out essential IN $(m+1)$-tuples.  Since $M_m$ is an
infinite minimal subshift, it is free.
\end{proof}
Let $H\triangleleft G$ be a finite-index normal subgroup, and let
$(X,H)$ be a free minimal uniquely ergodic zero-entropy
finite-alphabet $H$-subshift.  Define
\[
G\times_H X
:=
(G\times X)/\sim,
\]
where
\[
(gh,x)\sim(g,hx)
\qquad
(g\in G,\ h\in H,\ x\in X),
\]
and let $G$ act on $G\times_H X$ by
\[
a\cdot[g,x]:=[ag,x]
\qquad
(a,g\in G,\ x\in X).
\]
\begin{lemma}
\label{lem:finite-index-induction}
 Put
$
    \widetilde X:=G\times_H X.
$
Then $(\widetilde X,G)$ is conjugate to a free minimal uniquely ergodic
zero-entropy finite-alphabet $G$-subshift.

Moreover, if $X$ admits an essential IT $m$-tuple and no essential
IN $(m+1)$-tuple, then the same is true for $\widetilde X$.
\end{lemma}

\begin{proof}
Choose a finite set $R\subseteq G$ of representatives for $G/H$ with $e\in R$.  Then
\[
\widetilde X
=
\bigsqcup_{r\in R}X_r,
\qquad
X_r:=\{[r,x]:x\in X\},
\]
and each $X_r$ is a clopen copy of $X$.
The group $G$ acts transitively on these copies, while the
stabilizer of the identity copy
$
X_e=\{[e,x]:x\in X\}
$
is $H$, whose action there is naturally conjugate to the original
action on $X$. 
Therefore, minimality of $(X,H)$ implies minimality of $(\widetilde X,G)$.
Moreover, if
$
a[g,x]=[g,x],
$
then there exists $h\in H$ such that
\[
ag=gh
\qquad\text{and}\qquad
hx=x.
\]
The freeness of the $H$-action gives $h=e$, and hence $a=e$.
Thus the induced $G$-action is free.

Since $X\subseteq\mathcal A^H$ for some finite alphabet
$\mathcal A$, we may choose a symbol $\#\notin\mathcal A$, and  define
\[
\Psi([g,x])(k)
:=
\begin{cases}
x(g^{-1}k),& k\in gH,\\
\#,& k\notin gH.
\end{cases}
\]
Then $\Psi$ is a well-defined injective continuous $G$-equivariant map
from $\widetilde X$ into $(\mathcal A\cup\{\#\})^G$.  Hence
$\Psi(\widetilde X)$ is a finite-alphabet $G$-subshift conjugate to
$(\widetilde X,G)$.

Let $\widetilde\mu$ be a $G$-invariant probability measure on
$\widetilde X$.  Consider the natural factor
$
    q:\widetilde X\to G/H.
$ Its projection via $q$ to the finite transitive system $G/H$
is necessarily the uniform measure.  Hence every component has
$\widetilde\mu$-measure $[G:H]^{-1}$.  The normalized restriction of
$\widetilde\mu$ to the identity component $X_e$ is $H$-invariant and
therefore, by unique ergodicity of $(X,H)$, is uniquely determined.
The measures on the remaining components are then determined by
$G$-invariance.  Thus $(\widetilde X,G)$ is uniquely ergodic.

We next verify that the induced system has zero topological entropy.
Since $H\triangleleft G$, each clopen copy $X_r$, $r\in R$, is
$H$-invariant.  Under the identification
$
X\to X_r,$ $x\mapsto[r,x]$
the restricted $H$-action on $X_r$ is obtained from the original
$H$-action on $X$ by the automorphism
$
h\mapsto r^{-1}hr.
$
Hence
\[
h_{\mathrm{top}}(X_r,H)=0
\qquad\text{for every }r\in R.
\]
Since
$
\widetilde X=\bigsqcup_{r\in R}X_r
$
is a finite union of closed $H$-invariant subsets, the finite-union
property of topological entropy gives
$
h_{\mathrm{top}}(\widetilde X,H)=0.
$
Finally, the finite-index formula
$
h_{\mathrm{top}}(\widetilde X,H)
=
[G:H]\,h_{\mathrm{top}}(\widetilde X,G)
$
yields
\[
h_{\mathrm{top}}(\widetilde X,G)=0.
\]

It remains to compare independence. 
Since $G/H$ is finite, it is null, and so every IN-tuple, and therefore every
IT-tuple, is contained in a single component.  After translating, we
may assume that this component is the identity component $X$.
Let $U_1,\ldots,U_N$ be neighborhoods contained in the identity
component $X_e$, and let $J\subseteq G$ be an independence set for
$(U_1,\ldots,U_N)$.  For $s,t\in J$ and $i,j\in[N]$, the independence
property gives
\[
s^{-1}U_i\cap t^{-1}U_j\ne\emptyset.
\]
Since
$
s^{-1}U_i\subseteq X_{s^{-1}H},
$ and $
t^{-1}U_j\subseteq X_{t^{-1}H},
$
and distinct components are disjoint, we must have
\[
s^{-1}H=t^{-1}H.
\]
Thus $J\subseteq Hg_0$ for some $g_0\in G$, and
$Jg_0^{-1}\subseteq H$ is an $H$-independence set of the same
cardinality.  Hence an essential IN $(m+1)$-tuple in $\widetilde X$
would give one in $X$.  Conversely, every $H$-independence set is also
a $G$-independence set, so an essential IT $m$-tuple in $X$ remains
one in $\widetilde X$.
\end{proof}

\begin{proof}[Proof of Theorem~\ref{thm:intro-unique-virtual-abelian}]
Let $H\triangleleft G$ be a finite-index normal subgroup with
$H\cong\mathbb Z^p$.  We first construct a free minimal uniquely
ergodic zero-entropy $H$-subshift with the required independence
properties, and then apply Lemma~\ref{lem:finite-index-induction}
to induce it to a $G$-subshift.

Let $(M_m,\sigma)$ be the generalized $m$-Morse subshift, and let $(S,\sigma_S)$ be a Sturmian subshift.  Recall that
$(S,\sigma_S)$ is free, minimal, uniquely ergodic, null, and has zero
topological entropy.  Put
\[
Z_{m,p}:=M_m\times S^{p-1},
\]
with the convention that $S^0$ is a point, and let
$H\cong\mathbb Z^p$ act coordinatewise:
\[
(n_1,\ldots,n_p)
(x,y_2,\ldots,y_p)
:=
\bigl(
\sigma^{n_1}x,
\sigma_S^{n_2}y_2,\ldots,
\sigma_S^{n_p}y_p
\bigr).
\]
With the coordinatewise $\mathbb Z^p$-action, $Z_{m,p}$ is naturally
conjugate to a finite-alphabet $\mathbb Z^p$-subshift.

Under the coordinatewise $\mathbb Z^p$-action, the orbit of
$(x,y_2,\ldots,y_p)$ is the product of the corresponding
one-dimensional orbits.  Hence minimality of the factors implies
minimality of $Z_{m,p}$.  Moreover, successive disintegration with
respect to the coordinate projections shows that every
$\mathbb Z^p$-invariant probability measure is
$
\mu_m\otimes\mu_S^{\otimes(p-1)},
$
so $Z_{m,p}$ is uniquely ergodic.  Finally, freeness follows
coordinatewise from the freeness of the one-dimensional factors.

Its unique invariant measure is
$\mu_m\otimes\mu_S^{\otimes(p-1)}$, which has zero entropy for the
coordinatewise $\mathbb Z^p$-action.  The variational principle
therefore gives
\[
h_{\mathrm{top}}(Z_{m,p},H)=0.
\]

We next determine the independence order of $Z_{m,p}$.
By Lemma~\ref{lem:Morse-exact-degree-independence},
$M_m$ admits an essential IT $m$-tuple.  Fixing arbitrary points in
the Sturmian coordinates and using the corresponding independence
set along the first coordinate axis
$
\{(n,0,\ldots,0):n\in J\}\subseteq\mathbb Z^p,
$
we obtain an essential IT $m$-tuple in $Z_{m,p}$.

On the other hand, suppose that
$
(z_1,\ldots,z_{m+1})
$
were an essential IN $(m+1)$-tuple in $Z_{m,p}$.  Its projection to
each Sturmian coordinate is an IN $(m+1)$-tuple.  Since the Sturmian
system is null, every such projected tuple is diagonal.  Hence all
$z_i$ have the same Sturmian coordinates.  Since the original tuple
is essential, their $M_m$-coordinates are therefore pairwise
distinct.  Projecting to the first coordinate would then give an
essential IN $(m+1)$-tuple in $M_m$, contradicting
Lemma~\ref{lem:Morse-exact-degree-independence}.  Thus
$(Z_{m,p},H)$ has no essential IN $(m+1)$-tuple.

We may now apply Lemma~\ref{lem:finite-index-induction} to
$H\triangleleft G$ and $(Z_{m,p},H)$.  It yields a free minimal
uniquely ergodic zero-entropy finite-alphabet $G$-subshift
$(X_{m,G},G)$ which admits an essential IT $m$-tuple and has no
essential IN $(m+1)$-tuple.  Hence, \eqref{eq:maximal-pattern-IN-rank} gives
\[
h_{\mathrm{top}}^*(X_{m,G},G)=\log m.
\]

Finally, if $G=\mathbb Z$, neither the product construction nor
finite-index induction is needed: one may simply take
$
X_{m,\mathbb Z}=M_m\subseteq\{0,1,\ldots,m-1\}^{\mathbb Z}.
$
\end{proof}
\section{Realization and sharpness of finite multiplicity profiles}
\label{sec:realization}

This section proves Theorem~\ref{thm:finite-profile-realization}: every
finite compact--residual multiplicity profile is realized by a
zero-entropy minimal almost one-to-one extension of an irrational circle
rotation. 
The proof of the realization theorem relies on two  powerful results.
First, the strengthened construction of Haupt and J\"ager provides,
over one common irrational circle rotation, totally strictly ergodic
blocks with arbitrary prescribed finite residual multiplicity; see
\cite[Section~5.1]{HauptJaeger2025}.  Second, the relative
Furstenberg--Weiss theorem of Downarowicz and Lacroix allows finitely many
such blocks to be amalgamated into a single minimal almost one-to-one
extension while preserving the invariant-measure simplex and the
corresponding measure-preserving systems; see
\cite[Theorems~2 and~3]{DownarowiczLacroix1998}.

Throughout the section, we fix a finite multiset
\[
\mathcal P
=
\multiset{
(\iota_1,b_1),\ldots,(\iota_\ell,b_\ell)
}
\]
of pairs of positive integers.

\subsection{Multiplicity profiles and auxiliary results}
\label{subsec:profile-inputs}
We summarize in the following theorem the features of the construction of
Haupt and J\"ager \cite{HauptJaeger2025} that will be used below.
\begin{theorem}
\label{thm:common-blocks}
There exist an irrational number $\alpha\in\Torus$ and, for every
$b\in\N$, a $\mb Z$-system $(B_b,S_b)$, a topological factor
map
\[
q_b:(B_b,S_b)\to(Y,R_\alpha),
\qquad
Y:=\Torus,
\qquad
R_\alpha(y):=y+\alpha,
\]
and a unique invariant probability measure $\nu_b$ such that:
\begin{enumerate}[label=\textup{(\roman*)},leftmargin=2.5em]
\item
$(B_b,S_b)$ is totally strictly ergodic\footnote{A homeomorphism
$S$ is called \emph{totally strictly ergodic} if $S^n$ is minimal
and uniquely ergodic for every $n\in\mathbb Z\setminus\{0\}$.};

\item
$(Y,R_\alpha)$ is the maximal equicontinuous factor of
$(B_b,S_b)$;

\item
there exist measurable maps
$
\gamma_{b,s}:Y\to B_b,
$ $
s\in[b],
$
whose values are pairwise distinct for $m_Y$-a.e. $y$, such
that
\begin{equation}\label{eq:block-disintegration}
\nu_b
=
\int_Y(\nu_b)_y\,dm_Y(y),
\qquad
(\nu_b)_y
=
\frac1b\sum_{s=1}^b\delta_{\gamma_{b,s}(y)}\qquad\text{for $m_Y$-a.e. $y$;}
\end{equation}

\item  $q_b$ realizes the maximal compact
factor of $(B_b,\nu_b,S_b)$;

\item
$h_{\mathrm{top}}(S_b)=0$.
\end{enumerate}
\end{theorem}

\begin{proof}
Use the strengthened Anosov--Katok construction in
\cite[Section~5.1]{HauptJaeger2025}, which starts from a single
diffeomorphism over an irrational rotation $R_\alpha$ and makes all
iterates of all its finite lifts strictly ergodic.  Let $B_1$ be the
original system and, for $b\ge2$, let $B_b$ be its rescaled
$b$-fold lift.  Assertions \textup{(i)}, \textup{(ii)}, and
\textup{(iii)} follow from
\cite[Theorem~1.2 and Section~5, especially
equations~(5.5)--(5.6)]{HauptJaeger2025}

For $b=1$, the system is measure-theoretically isomorphic to the base
rotation by (iii).  For $b\ge2$,
\cite[Section~5.2]{HauptJaeger2025} shows that the lift has no
dynamical eigenvalues other than those of $R_\alpha$,
which proves \textup{(iv)}.

Finally, $q_b$ is a finite measure-theoretic extension of the
zero-entropy rotation, and therefore
$
h_{\nu_b}(S_b)=0.
$
Since $S_b$ is uniquely ergodic, the variational principle gives
$
h_{\mathrm{top}}(S_b)=0.
$
\end{proof}
The next lemma adds a cyclic compact coordinate of order $\iota$ to
the residual block $B_b$, thereby realizing the pair $(\iota,b)$.

\begin{lemma}
\label{lem:single-pair-realization}
Let $\iota,b\in\N$.  Put
\[
C_\iota:=\Z/\iota\Z,
\qquad
r_\iota(j):=j+1,
\]
and let $\upsilon_\iota$ be the uniform probability measure on
$C_\iota$.  Define
\[
\widehat X_{\iota,b}
:=
B_b\times C_\iota,
\qquad
\widehat T_{\iota,b}
:=
S_b\times r_\iota,
\qquad
\widehat\mu_{\iota,b}
:=
\nu_b\otimes\upsilon_\iota.
\]
Moreover, put
\[
Z_\iota:=Y\times C_\iota,
\qquad
m_\iota:=m_Y\otimes\upsilon_\iota,
\]
and define
\[
R_{\alpha,\iota}(y,j):=(R_\alpha(y),r_\iota(j)).
\]
Then the following statements hold.
\begin{enumerate}[label=\textup{(\roman*)},leftmargin=2.5em]
\item
The system
$(\widehat X_{\iota,b},\widehat T_{\iota,b})$
is strictly ergodic, with unique invariant probability measure
$\widehat\mu_{\iota,b}$.

\item
The map
\[
\widehat\kappa_{\iota,b}:
\widehat X_{\iota,b}\to Z_\iota,
\qquad
\widehat\kappa_{\iota,b}(x,j):=(q_b(x),j),
\]
is a factor map onto
$(Z_\iota,R_{\alpha,\iota})$ and represents the
measure-theoretic maximal compact factor of
$(\widehat X_{\iota,b},
\widehat\mu_{\iota,b},
\widehat T_{\iota,b})$.

\item
The disintegration of $\widehat\mu_{\iota,b}$ over
$\widehat\kappa_{\iota,b}$ is
\begin{equation}\label{eq:single-block-conditionals}
\widehat\mu_{\iota,b}
=
\int_{Y\times C_\iota}
\bigl((\nu_b)_y\otimes\delta_j\bigr)
\,d(m_Y\otimes\upsilon_\iota)(y,j).
\end{equation}
In particular, its conditional measures have exactly $b$ support
points for $m_\iota$-a.e. $(y,j)\in Z_\iota$.

\item
The factor map
\[
p_\iota:Z_\iota\to  Y,
\qquad
p_\iota(y,j):=y,
\]
has constant degree $\iota$, and
$
p_\iota\circ\widehat\kappa_{\iota,b}
=
q_b\circ\operatorname{pr}_{B_b}.
$

\item
$
h_{\mathrm{top}}(\widehat T_{\iota,b})=0.
$
\end{enumerate}
\end{lemma}

\begin{proof}
Since $S_b^\iota$ is strictly ergodic, the standard cyclic-extension
argument shows that
$
\widehat T_{\iota,b}=S_b\times r_\iota
$
is strictly ergodic, with unique invariant measure
$
\widehat\mu_{\iota,b}
=
\nu_b\otimes\upsilon_\iota.
$
Indeed, the return map to each fiber
$B_b\times\{j\}$ is $S_b^\iota$.
Since the finite cyclic system $C_\iota$ is compact,
$
\widehat\kappa_{\iota,b}
=
q_b\times\id_{C_\iota}
$
represents the measure-theoretic maximal compact factor of the product.
The disintegration over $\widehat\kappa_{\iota,b}$ is
\[
\widehat\mu_{\iota,b}
=
\int_{Y\times C_\iota}
\bigl((\nu_b)_y\otimes\delta_j\bigr)
\,d(m_Y\otimes\upsilon_\iota)(y,j).
\]
By \eqref{eq:block-disintegration}, these conditional measures have
exactly $b$ support points  $m_\iota$-a.e. $(y,j)\in Z_\iota$.  The map
$p_\iota:Z_\iota\to Y$ plainly has degree $\iota$.  Finally,
\[
h_{\mathrm{top}}(\widehat T_{\iota,b})
=
h_{\mathrm{top}}(S_b)+h_{\mathrm{top}}(r_\iota)
=
0.\qedhere
\]
\end{proof}

We also use the following relative model theorem of Downarowicz and
Lacroix \cite{DownarowiczLacroix1998}.  Recall that a Borel subset of a topological system is
\emph{full} if it has measure one for every invariant probability
measure.

\begin{theorem}
\label{thm:DL}
Let
$
(Y,R_\alpha)
$
be an irrational circle rotation, and let
$
\widehat\pi:
(\widehat X,\widehat T)
\to
(Y,R_\alpha)
$
be a topological factor map.  Then there exist a minimal $\mb Z$-system $(X,T)$, an almost one-to-one topological factor map
$
\pi:(X,T)\to(Y,R_\alpha),
$
full Borel sets
$
\widehat X_0\subseteq\widehat X,
$ and $
X_0\subseteq X,
$
and a Borel bijection
$
\Phi:\widehat X_0\to X_0
$
such that
\[
\Phi\circ\widehat T=T\circ\Phi
\qquad\text{and}\qquad
\pi\circ\Phi=\widehat\pi.
\]
Moreover, the induced map
$
\Phi_*:
\M_{\widehat T}(\widehat X)
\to
\M_T(X)
$
is an affine weak* homeomorphism.  In particular, for every
$\widehat\mu\in\M_{\widehat T}(\widehat X)$,
$
\Phi:
(\widehat X,\widehat\mu,\widehat T)
\to
(X,\Phi_*\widehat\mu,T)
$
is an isomorphism.
\end{theorem}
\begin{proof}
An irrational circle rotation is strictly ergodic, nonperiodic, and has
zero topological entropy.  The conclusion is therefore
\cite[Theorem~3]{DownarowiczLacroix1998}.  The assertions concerning
$\Phi_*$ and the corresponding measure-preserving systems are part of
the Borel$^*$-isomorphism conclusion; see
\cite[p.~2]{DownarowiczLacroix1998}.
\end{proof}

\subsection{Proof of Theorem~\ref{thm:finite-profile-realization}}
\label{subsec:profile-amalgamation}

\begin{proof}[Proof of Theorem~\ref{thm:finite-profile-realization}]
For each $i\in[\ell]$, apply
Lemma~\ref{lem:single-pair-realization} with
$(\iota,b)=(\iota_i,b_i)$, and write
\[
(\widehat X_i,\widehat T_i,\widehat\mu_i)
:=
(\widehat X_{\iota_i,b_i},
 \widehat T_{\iota_i,b_i},
 \widehat\mu_{\iota_i,b_i}),
\qquad
Z_i:=Z_{\iota_i},
\qquad
m_i:=m_{\iota_i}.
\]
Let
$
\widehat\kappa_i:\widehat X_i\to Z_i,
$ and $
p_i:Z_i\to Y
$
be the corresponding factor maps, and put
$\widehat\pi_i:=p_i\circ\widehat\kappa_i$.
By Lemma~\ref{lem:single-pair-realization},
$(\widehat X_i,\widehat T_i)$ is strictly ergodic,
$\widehat\kappa_i$ is its  maximal compact
factor map, every fiber of $p_i$ has cardinality $\iota_i$, the
conditional measures of $\widehat\mu_i$ over
$\widehat\kappa_i$ have exactly $b_i$ support points almost
everywhere, and
$
h_{\mathrm{top}}(\widehat T_i)=0.
$

Form the finite disjoint union
\[
\widehat X:=\bigsqcup_{i=1}^{\ell}\widehat X_i,
\qquad
\widehat T|_{\widehat X_i}:=\widehat T_i,
\qquad
\widehat\pi|_{\widehat X_i}:=\widehat\pi_i.
\]
Then
$
\widehat\pi:
(\widehat X,\widehat T)\to(Y,R_\alpha)
$
is a topological factor map.  Since the sets $\widehat X_i$ are
clopen and invariant and each block is uniquely ergodic,
\[
\M_{\widehat T}^e(\widehat X)
=
\{\widehat\mu_1,\ldots,\widehat\mu_\ell\}.
\]

Applying Theorem~\ref{thm:DL} to $\widehat\pi$, we obtain a minimal
$\mb Z$-system $(X,T)$, an almost one-to-one factor map
$
\pi:(X,T)\to(Y,R_\alpha),
$
full Borel sets
$\widehat X_0\subseteq\widehat X$ and $X_0\subseteq X$, and a
Borel bijection
$
\Phi:\widehat X_0\to X_0
$
such that
\[
\Phi\circ\widehat T=T\circ\Phi,
\qquad
\pi\circ\Phi=\widehat\pi.
\]
Moreover, $\Phi_*$ is an affine bijection between the invariant
measure spaces.  It therefore maps extreme points bijectively onto
extreme points, and hence
\[
\M_T^e(X)=\{\mu_1,\ldots,\mu_\ell\},
\qquad
\mu_i:=\Phi_*\widehat\mu_i.
\]

Since $\pi:X\to Y$ is almost one-to-one and $Y$ is equicontinuous, one has
$\pi$ is the maximal equicontinuous factor map of $(X,T)$;
see, for example,
\cite[Proposition~1.1]{GabrielDominik2020}.

It remains to identify the compact and residual multiplicities.  Fix
$i\in[\ell]$, and put
$
X_{0,i}
:=
\Phi(\widehat X_0\cap\widehat X_i).
$
By the Lusin--Souslin theorem
\cite[Theorem~2.8\textup{(2)}]{Glasner2003}, the set $X_{0,i}$ is
Borel, and the restriction
$
\Phi_i:
\widehat X_0\cap\widehat X_i\to X_{0,i}
$
is a measure-theoretic isomorphism from
$(\widehat X_i,\widehat\mu_i,\widehat T_i)$ to
$(X,\mu_i,T)$.  Moreover, it is an isomorphism over $Y$, since
\[
\pi\circ\Phi_i
=
\widehat\pi_i
=
p_i\circ\widehat\kappa_i
\qquad
\widehat\mu_i\text{-a.e.}
\]

Define, modulo a $\mu_i$-null set,
$
\kappa_i
:=
\widehat\kappa_i\circ\Phi_i^{-1}.
$
Then
$\kappa_i$ is the measure-theoretic maximal compact factor of
$(X,\mu_i,T)$.  Furthermore,
\[
\pi=p_i\circ\kappa_i
\qquad
\mu_i\text{-a.e.}
\]
Hence
\[
\iota_{\mu_i}
=
|p_i^{-1}(y)|
=
\iota_i.
\]

Maximal measure sequence entropy is invariant under
measure-theoretic isomorphism.  Therefore, by
Proposition~\ref{prop:residual-sequence-entropy},
\[
b_{\mu_i}
=
\exp\bigl(h_{\mu_i}^*(T)\bigr)
=
\exp\bigl(h_{\widehat\mu_i}^*(\widehat T_i)\bigr)
=
b_i.
\]
Thus
\[
\M_T^e(X)=\{\mu_1,\ldots,\mu_\ell\},
\qquad
\Prof(X,T)=\mathcal P.
\]

Finally,
\[
h_{\mu_i}(T)
=
h_{\widehat\mu_i}(\widehat T_i)
=
0
\qquad
(i\in[\ell]).
\]
Since these are all the ergodic invariant measures of $X$, the
variational principle gives
\[
h_{\mathrm{top}}(T)=0.\qedhere
\]
\end{proof}

The realization theorem yields the following sharp characterization.
\begin{corollary}
\label{cor:divisibility}
Let $d,b\in\N$.  The following are equivalent.
\begin{enumerate}[label=\textup{(\roman*)},leftmargin=2.4em]
\item
There exists a uniquely ergodic zero-entropy minimal almost one-to-one
extension $(X,T)$ of an irrational circle rotation, with unique
invariant measure $\mu$, such that
\[
\tdeg_{\mathbb Z}(X)=d
\qquad\text{and}\qquad
h_\mu^*(T)=\log b.
\]

\item
$b\mid d$; that is, $d=kb$ for some $k\in\N$
\end{enumerate}
\end{corollary}

\begin{proof}
Suppose that \textup{(i)} holds.  By
Theorem~\ref{thm:intro-degree},
\[
d
=
\tdeg_{\mathbb Z}(X)
=
\iota_\mu
\exp\bigl(h_\mu^*(T)\bigr)
=
\iota_\mu b.
\]
Since $\iota_\mu\in\N$, it follows that $b\mid d$.

Conversely, suppose that $b\mid d$.  Applying
Theorem~\ref{thm:finite-profile-realization} to the singleton profile
$
\multiset{(d/b,b)}
$
gives a uniquely ergodic zero-entropy minimal almost one-to-one
extension of an irrational circle rotation satisfying
$
\tdeg_{\mathbb Z}(X)
=
(d/b)b
=
d
$
and
$
h_\mu^*(T)=\log b.
$
\end{proof}
\appendix

\section{Conditional degree and finite topomorphic extensions}
\label{app:finite-topomorphic-characterization}
Throughout this appendix, let $G$ be a countable discrete group and
let $(X,G)$ be a minimal $G$-system with
$\M_G(X)\ne\emptyset$.  

\begin{definition}
\label{def:finite-topomorphic-extension}
Let $m\in\N$.  We say that
$
\pi:(X,G)\to(\Xeq,G)
$
is an \emph{at most $m$-to-one topomorphic extension} if there exist
Borel maps
$
\gamma_1,\ldots,\gamma_m:\Xeq\to X
$
such that
$
\pi\circ\gamma_j=\id_{\Xeq},
$ $
j=1,\ldots,m,
$
and
$
\mu(\Gamma_\gamma)=1$ for every
$\mu\in\M_G(X),
$
where
$
\Gamma_\gamma
:=
\bigcup_{j=1}^m
\{x\in X:x=\gamma_j(\pi(x))\}.
$
Repetitions among the maps $\gamma_j$ are allowed.
\end{definition}

\begin{proposition}
\label{prop:topomorphic-characterization}
Let $m\in\N$.  Then
$
\pi:(X,G)\to(\Xeq,G)
$
is an at most $m$-to-one topomorphic extension if and only if
$
\tdeg_G(X)\le m.
$
\end{proposition}

\begin{proof}
Suppose first that $\pi$ is an at most $m$-to-one topomorphic
extension.  For $\mu\in\M_G(X)$, write
$
\mu=\int_{\Xeq}\mu_y\,d\nu(y).
$
Since $\mu(\Gamma_\gamma)=1$, one has
\[
\mu_y
\bigl(
\{\gamma_1(y),\ldots,\gamma_m(y)\} 
\bigr)
=1\qquad\text{for $\nu$-a.e. $y\in\Xeq$.}
\]
  Hence
\[
\supp(\mu_y)
\subseteq
\{\gamma_1(y),\ldots,\gamma_m(y)\}\qquad\text{for $\nu$-a.e. $y\in\Xeq$}
\]
 and therefore
\[
\tdeg_G(X)\le m.
\]

Conversely, suppose that
$
d:=\tdeg_G(X)\le m.
$
By Proposition~\ref{prop:finite-conditional-degree}, there are finitely
many ergodic measures
$
\M_G^e(X)=\{\mu_1,\ldots,\mu_\ell\}
$
and Borel maps
\[
\gamma_{i,j}:\Xeq\to X,
\qquad
1\le i\le\ell,\quad 1\le j\le k_i,
\]
such that
\[
\sum_{i=1}^{\ell}k_i=d,
\qquad
\pi\circ\gamma_{i,j}=\id_{\Xeq},
\quad\text{and}\quad
\supp((\mu_i)_y)
=
\{\gamma_{i,1}(y),\ldots,\gamma_{i,k_i}(y)\}\quad\text{for $\nu$-a.e. $y\in\Xeq$.}
\]

Relabel these maps as
$
\gamma_1,\ldots,\gamma_d,
$
and, if $d<m$, repeat some of them to obtain
$\gamma_1,\ldots,\gamma_m$.  Every invariant measure is a convex
combination of $\mu_1,\ldots,\mu_\ell$, so its conditional measures
are supported on
$
\{\gamma_1(y),\ldots,\gamma_m(y)\}
$
for $\nu$-a.e. $y$.  Thus
\[
\mu(\Gamma_\gamma)=1
\qquad
\text{for every }\mu\in\M_G(X),
\]
and $\pi$ is an at most $m$-to-one topomorphic extension.
\end{proof}

\end{document}